\documentclass[12pt]{amsart}
\usepackage{amsmath,amssymb,amsfonts}
\usepackage{mathtools}
\usepackage[shortlabels]{enumitem}
\setlist{leftmargin=*}

\usepackage[colorlinks=true, linkcolor=blue]{hyperref}

\usepackage[sorted]{amsrefs}

\newtheorem{thm}{Theorem}[section]
\newtheorem*{thm*}{Theorem}
\newtheorem{cor}[thm]{Corollary}
\newtheorem{lem}[thm]{Lemma}
\newtheorem{prop}[thm]{Proposition}
\newtheorem{prob}[thm]{Problem}

\newtheorem{claim}[thm]{Claim}
\newtheorem*{claim*}{Claim}

\newtheorem{fact}[thm]{Fact}

\theoremstyle{definition}
\newtheorem{defn}[thm]{Definition}
\theoremstyle{remark}

\newtheorem{rem}[thm]{Remark}

\newtheorem{sample}[thm]{Example} \numberwithin{equation}{section}
    {\medskip\begingroup\leftskip 0.5cm\rightskip 0.5cm\noindent\begin{small}{\bf Remark.}}
    {\end{small}\par\endgroup}
{\begin{list}{$\bullet$}
 {\settowidth{\labelwidth}{\textsf{$\bullet$}} \setlength{\leftmargin}{10pt}}}
{\end{list}}

\newcounter{ssample}[section]

{\color{Bittersweet} \noindent Example \refstepcounter{ssample}\hbox{\bf \arabic{section}.\arabic{ssample}.}}
{}

\newcounter{insertcount}

    {\color{blue} \medskip\begingroup\noindent\begin{small}{\color{blue} \stepcounter{insertcount}
          {
            \bf Insert \arabic{insertcount}. #1.}
            \addcontentsline{toc}{subsection}{{\ \ \small  Insert \arabic{insertcount}: #1}}
               \leavevmode  }
           }
    {\end{small}\par\endgroup}

\newcommand{\mrmk}[1]
{{\tiny$^{\spadesuit}$}\marginpar{\fbox{\footnotesize #1}}}
\def\strutdepth{\dp\strutbox}%
\def\marginalnote#1{\strut\vadjust{\kern-\strutdepth\specialnote{#1}}}%
\def\specialnote#1{\vtop to \strutdepth{\baselineskip%
\strutdepth\vss\llap{\hbox{\scriptsize \bf #1}}\null}}%

\newcommand{\RR}{\mathbb{R}}

\newcommand{\NN}{\mathbb N}

\reversemarginpar

\def\acl{\operatorname{acl}}

\newcommand*\bbar[1]{%
  \hbox{%
    \vbox{%
      \hrule height 0.5pt 
      \kern0.5ex
      \hbox{%
        \kern-0.1em
        \ensuremath{#1}%
        \kern-0.1em
      }%
    }%
  }%
}

\def\Ind#1#2{#1\setbox0=\hbox{$#1x$}\kern\wd0\hbox to 0pt{\hss$#1\mid$\hss}
\lower.9\ht0\hbox to 0pt{\hss$#1\smile$\hss}\kern\wd0}

\def\ind{\mathop{\mathpalette\Ind{}}}

\def\notind#1#2{#1\setbox0=\hbox{$#1x$}\kern\wd0
\hbox to 0pt{\mathchardef\nn=12854\hss$#1\nn$\kern1.4\wd0\hss}
\hbox to 0pt{\hss$#1\mid$\hss}\lower.9\ht0 \hbox to 0pt{\hss$#1\smile$\hss}\kern\wd0}

\title[Zarankiewicz's problem for semilinear hypergraphs] {Zarankiewicz's problem for semilinear hypergraphs}

\author[A. Basit]{Abdul Basit} 
\address{Department of Mathematics\\
Iowa State University\\
Ames, IA, 50011, USA}
\email{abasit@iastate.edu}

\author[A. Chernikov]{Artem Chernikov} 
\address{Department of Mathematics\\
University of California Los Angeles\\
Los Angeles, CA 90095-1555} 
\email{chernikov@math.ucla.edu}

\author[S. Starchenko]{Sergei Starchenko}
\address{Department of Mathematics\\
University of Notre Dame\\
Notre Dame, IN, 46656, USA}
\email{sstarche@nd.edu}

\author[T.Tao]{Terence Tao} 
\address{Department of Mathematics\\
University of California Los Angeles\\
Los Angeles, CA 90095-1555} 
\email{tao@math.ucla.edu}

\author[C.-M.Tran]{Chieu-Minh Tran}
\address{Department of Mathematics\\
University of Notre Dame\\
Notre Dame, IN, 46656, USA}
\email{mtran6@nd.edu}

\begin{document}
\maketitle
\begin{abstract}
A bipartite graph $H = \left(V_1, V_2; E \right)$ with $|V_1| + |V_2| = n$ is \emph{semilinear} if $V_i \subseteq \mathbb{R}^{d_i}$ for some $d_i$ and the edge relation $E$ consists of the pairs of points $(x_1, x_2) \in V_1 \times V_2$ satisfying a fixed Boolean combination of $s$ linear equalities and inequalities in $d_1 + d_2$ variables for some $s$. We show that for a fixed $k$, the number of edges in a $K_{k,k}$-free semilinear $H$ is almost linear in $n$, namely $|E| = O_{s,k,\varepsilon}(n^{1+\varepsilon})$ for any $\varepsilon > 0$; and more generally, $|E| = O_{s,k,r,\varepsilon}(n^{r-1 + \varepsilon})$ for a $K_{k, \ldots,k}$-free semilinear $r$-partite $r$-uniform hypergraph.

As an application, we obtain the following incidence bound: given $n_1$ points and  $n_2$ open boxes with axis parallel sides in $\mathbb{R}^d$ such that their incidence graph is $K_{k,k}$-free, there can be at most $O_{k,\varepsilon}(n^{1+\varepsilon})$ incidences. The same bound holds if instead of boxes one takes polytopes cut out by the translates of an arbitrary fixed finite set of halfspaces.

	We also obtain matching upper and (superlinear) lower bounds in the case of dyadic boxes on the plane, and point out some connections to the model-theoretic trichotomy in $o$-minimal structures (showing that the failure of an almost linear bound for some definable graph allows one to recover the field operations from that graph in a definable manner).
\end{abstract}

\section{Introduction}

We fix $r \in \mathbb{N}_{\geq 2}$ and let $H = \left(V_1, \ldots, V_r; E \right)$ be an $r$-partite and $r$-uniform hypergraph (or just \emph{$r$-hypergraph} for brevity) with vertex sets $V_1, \ldots, V_r$ having  $|V_i| = n_i$, (hyper-) edge set $E$, and $n = \sum_{i=1}^r n_i$ being the total number of vertices.

\emph{Zarankiewicz's problem} asks
for the maximum number of edges in such a hypergraph $H$ (as a function of $n_1, \ldots, n_r$) assuming that it does not contain the complete $r$-hypergraph $K_{k, \ldots, k}$ with $k > 0$ a fixed number of vertices in each part. 
The following classical upper bound is due to K\H{o}v\'ari, S\'os and Tur\'an \cite{kovari1954problem} for $r=2$ and Erd\H os \cite{erdos1964extremal} for a general $r$: if $H$ is $K_{k, \ldots, k}$-free, then $|E| = O_{r,k} \left(n^{r - \frac{1}{k^{r-1}}} \right)$. A probabilistic construction in \cite{erdos1964extremal} also shows  that the exponent cannot be substantially improved.

However, stronger bounds are known for restricted families of  hypergraphs arising in geometric settings. For example, if $H$ is  the incidence graph of a set of $n_1$ points and $n_2$ lines in $\RR^2$, then $H$ is $K_{2,2}$-free, and K\H{o}v\'ari-S\'os-Tur\'an Theorem implies $|E| = O( n^{3/2})$. The Szemer\'{e}di-Trotter Theorem~\cite{szemeredi1983extremal} improves this and gives the  optimal bound $|E| = O(n^{4/3})$. 
More generally, \cite{fox2017semi} gives improved bounds for {\em semialgebraic} graphs of bounded description complexity. 
This is generalized to semialgebraic hypergraphs in \cite{do2018zarankiewicz}. In a different direction, the results in~\cite{fox2017semi} are generalized to graphs definable in $o$-minimal structures in \cite{basu2018minimal} and, more generally, in distal structures in \cite{chernikov2020cutting}.

A related highly nontrivial problem is to understand when the bounds offered by the results in the preceding paragraph are sharp. When $H$ is the incidence graph of $n_1$ points and $n_2$ circles of unit radius in $\RR^2$, the best known upper bound is $|E| =O(n^{4/3})$, proven in~\cite{spencer1984unit} and also implied by the general bound for semialgebraic graphs. Any improvement to this bound will be a step toward resolving the long standing unit distance conjecture of Erd\H{o}s (an almost linear bound of the form $|E| =O(n^{1+c/\log\log n})$ will positively resolve it). 

This paper was originally motivated by the following incidence problem. Let $H$ be the incidence graph of a set of $n_1$ points and a set of $n_2$ solid rectangles \emph{with axis-parallel sides} (which we refer to as \emph{boxes}) in $\mathbb{R}^2$. Assuming that $H$ is $K_{2,2}$-free, i.e.~no two points belong to two rectangles simultaneously, what is the maximum number of incidences $|E|$? In the following theorem, we obtain an almost linear bound (which is much stronger than the bound implied by the aforementioned general result for semialgebraic graphs) and demonstrate that it is close to optimal.
\vspace{-1em}
\begin{thm*}[A]
\begin{enumerate}
	\item For any set $P$ of $n_1$ points in $\mathbb{R}^2$ and any set $R$ of $n_2$  boxes in $\mathbb{R}^2$, if the incidence graph on $P \times R$ is $K_{k,k}$-free, then it contains at most $O_k \left(n \log^{4}(n) \right)$ incidences (Corollary \ref{cor: inc with boxes} with $d=2$).
	\item If all boxes in $R$ are \emph{dyadic} (i.e.~direct products of intervals of the form $[s2^t, (s+1)2^t)$ for some integers $s,t$), then the number of incidences is at most $O_k \left( n \frac{\log(100+n_1)}{\log\log(100+n_1)} \right)$ (Theorem \ref{thm: upper bound for dyadic rects}).
	\item For an arbitrarily large  $n$,  there exist a set of $n$ points and $n$ dyadic boxes in $\mathbb{R}^2$ so that the incidence graph is $K_{2,2}$-free and the number of incidences is $\Omega \left(n \frac{\log(n)}{\log\log(n)} \right)$ (Proposition \ref{prop: lower bound for boxes}).
\end{enumerate}
\end{thm*}

\begin{prob}
While the bound for dyadic boxes is tight, we leave it as an open problem to close the gap between the upper and  the lower  bounds for arbitrary boxes.	
\end{prob}

\begin{rem}
 A related result in \cite{fox2008separator} demonstrates that every $K_{k,k}$-free intersection graph of $n$ convex sets on the plane satisfies $|E| = O_{k}(n)$. Note that in Theorem (B) we consider a $K_{k,k}$-free \emph{bipartite} graph, so in particular there is no restriction on the intersection graph of the boxes in $R$. 
\end{rem}

Theorem~(A.1) admits the following generalization to higher dimensions and more general polytopes.
\begin{thm*}[B]
\begin{enumerate}
	\item For any set $P$ of $n_1$ points and any set $B$ of $n_2$ boxes in  $\mathbb{R}^d$, if the incidence graph on $P \times B$ is $K_{k,k}$-free, then it contains at most $O_{d,k} \left( n  \log^{2 d} n \right)$ incidences (Corollary \ref{cor: inc with boxes}).
	\item More generally, given finitely many half-spaces $H_1, \ldots, H_s$ in $\mathbb{R}^d$, let $\mathcal{F}$ be the family of all possible polytopes in $\mathbb{R}^d$ cut out by arbitrary translates of $H_1, \ldots, H_s$. Then for any set $P$ of $n_1$ points in $\mathbb{R}^d$ and any set $F$ of $n_2$ polytopes in $\mathcal{F}$, if the incidence graph on $P \times F$ is $K_{k,k}$-free, then it contains at most $O_{k,s}\left( n \log^{s} n \right)$ incidences (Corollary \ref{cor: inc with polytopes}).
\end{enumerate}
\end{thm*}

\begin{prob}\label{prob: log power}
What is the optimal bound on the power of $\log n$ in Theorem (B)? In particular, does it actually have to grow with the dimension $d$?
\end{prob}

\begin{rem}
	A bound similar to Theorem (B.1) and an improved bound for Theorem (A.1) in the $K_{2,2}$-free case are established independently by Tomon and Zakharov in \cite{Tomon}, in which the authors also use our Theorem (A.3) to provide a counterexample to a conjecture of Alon
et al.~\cite{alon2015separation} about the number of edges in a graph of bounded separation dimension, as well as to a conjecture of Kostochka from \cite{kostochka2004coloring}. Some further Ramsey properties of semilinear graphs are demonstrated by Tomon in \cite{tomon2021ramsey}.
\end{rem}

The upper bounds in Theorems (A.1) and (B) are obtained as immediate applications of a general upper bound for Zarankiewicz's problem for \emph{semilinear} hypergraphs of bounded description complexity.

\begin{defn} \label{def: semilin sets}
Let $V$ be an ordered vector space over an ordered division ring $R$ (e.g.~$\mathbb{R}$ viewed as a vector space over itself). A set $X \subseteq V^d$ is \emph{semilinear, of description complexity $(s,t)$} if $X$ is  a union of at most $t$ sets of the form 
\[
\left\{ \bar{x}\in V^{d}:f_{1}\left(\bar{x}\right) \leq 0, \ldots, f_{p} \left(\bar{x}\right) \leq 0,f_{p+1}\left(\bar{x}\right)<0,\ldots,f_{s}\left(\bar{x}\right)<0\right\} \mbox{,}
\]
where $p \leq s \in \mathbb{N}$ and each $f_i: V^d \to V$ is a \emph{linear} function, i.e.,~of the form 
$$f\left(x_{1},\ldots,x_{d}\right)=\lambda_{1}x_{1}+\ldots+\lambda_{d}x_{d}+a$$
for some $\lambda_{i}\in R$ and $a\in V$.
	
\end{defn}
We focus on the case $V=R = \mathbb{R}$ in the introduction, in which case these are precisely the semialgebraic sets that can be defined using only \emph{linear} polynomials.

\begin{rem}\label{rem: semilin iff def}
	By a standard quantifier elimination result \cite[\S 7]{van1998tame}, every set definable in an ordered
  vector space over an ordered division ring, in the sense of model theory, is semilinear (equivalently, a projection of a semilinear set is a finite union of semilinear sets).
\end{rem}

\begin{defn}\label{def: semilin hypergraphs}
We say that an $r$-hypergraph $H$ is \emph{semilinear, of description complexity $(s,t)$} if there exist some $d_i \in \mathbb{N}, V_i \subseteq \mathbb{R}^{d_i}$ and a semilinear set $X \subseteq \mathbb{R}^d = \prod_{i \in [r]} \mathbb{R}^{d_i}$ of description complexity $(s,t)$ so that $H$ is isomorphic to the $r$-hypergraph $\left(V_1, \ldots, V_r; X \cap \prod_{i \in [r]} V_i \right)$.
\end{defn}

We stress that there is no restriction on the dimensions $d_i$ in this definition.
We obtain the following general upper bound.
\begin{thm*}[C]
If $H$ is a semilinear $r$-hypergraph of description complexity $(s,t)$ and $H$ is $K_{k, \ldots, k}$-free, then 
 $$|E| = O_{r,s,t,k} \left( n^{r-1} \log^{s(2^{r-1}-1)}(n) \right).$$
\end{thm*}
 In particular $|E| = O_{r,s,t,k,\varepsilon} \left( n^{r-1 + \varepsilon} \right)$ for any $\varepsilon>0$ in this case. For a more precise statement, see Corollary \ref{cor: semilin Zar} (in particular, the dependence of the constant in $O_{r,s,t,k}$ on $k$ is at most linear).
 
 \begin{rem}
 It is demonstrated in \cite{mustafa2015zarankiewicz} that a similar bound holds in the situation when $H$ is the intersection hypergraph of $(d-1)$-dimensional simplices in $\mathbb{R}^d$.
 \end{rem}
 
 \medskip

One can get rid of the logarithmic factor entirely by restricting to the family of all finite $r$-hypergraphs induced by a given $K_{k, \ldots, k}$-free semilinear relation (as opposed to all $K_{k, \ldots, k}$-free $r$-hypergraphs induced by a given arbitrary semilinear relation as in Theorem (C)).
\begin{thm*}[D]
	Assume that $X \subseteq \mathbb{R}^d = \prod_{i \in [r]} \mathbb{R}^{d_i}$ is semilinear and $X$  does not contain the direct product of $r$ infinite sets (e.g.~if $X$ is $K_{k, \ldots, k}$-free for some $k$). Then for any $r$-hypergraph $H$ of the form $\left(V_1, \ldots, V_r; X \cap \prod_{i \in [r]} V_i \right)$ for some finite $V_i \subseteq \mathbb{R}^{d_i}$, we have $|E| = O_X(n^{r-1})$.
\end{thm*}

This is Corollary \ref{cor: semilin wo log} and follows from a more general Theorem \ref{thm: non-cart lin} connecting linear Zarankiewicz bounds to a model-theoretic notion of \emph{linearity} of a first-order structure (in the sense that the matroid given by the algebraic closure operator behaves like the linear span in a vector space, as opposed to the algebraic closure in an algebraically closed field --- see Definition \ref{def: weak loc mod}).

In particular, for every $K_{k,k}$-free semilinear relation $X \subseteq \mathbb{R}^{d_1} \times \mathbb{R}^{d_2}$ (equivalently, $X$ definable with parameters in the first-order structure $(\mathbb{R}, <, +)$ by Remark \ref{rem: semilin iff def}) we have $|X \cap (V_1 \times V_2)| = O(n)$ for all $V_i \subseteq \mathbb{R}^{d_i}_i, |V_i| = n_i, n = n_1 + n_2$. One the other hand, by optimality of the Szemer\'{e}di-Trotter bound, for the semialgebraic $K_{2,2}$-free point-line incidence graph $X = \{(x_1,x_2; y_1, y_2) \in \mathbb{R}^4 : x_2 = y_1 x_1 + y_2 \}\subseteq \mathbb{R}^2 \times  \mathbb{R}^2$ we have $|X \cap (V_1 \times V_2)| = \Omega(n^{\frac{4}{3}})$.
Note that in order to define  it  we use both addition and multiplication, i.e.~the field structure. 
This is not coincidental --- as a consequence of the trichotomy theorem in $o$-minimal structures \cite{peterzil1998trichotomy}, we observe that the failure of a linear Zarankiewicz bound always allows to recover the field in a definable way (Corollary \ref{cor: lin iff loc mod}). In the semialgebraic case, we have the following corollary that is easy to state (Corollary \ref{cor: rec mult from failed zar}).
\begin{thm*}[E]
	Assume that $X \subseteq \mathbb{R}^d = \prod_{i \in [r]} \mathbb{R}^{d_i}$ for some $r,d_i \in \mathbb{N}$ is semialgebraic and $K_{k, \ldots, k}$-free, but $|X \cap \prod_{i \in [r]} V_i| \neq O(n^{r-1})$.
		 Then the graph of multiplication $\times \restriction_{[0,1]}$ restricted to the unit box is definable in $(\mathbb{R}, <, +, X)$.
\end{thm*}

We conclude with a brief overview of the paper.

In Section \ref{sec: upper bound} we introduce a more general class of hypergraphs definable in terms of \emph{coordinate-wise monotone} functions (Definition \ref{def: coord-wise mon funct}) and prove an upper Zarankiewicz bound for it (Theorem \ref{thm:main}). Theorems  (A.1), (B) and (C) are then deduced from it in Section \ref{sec: appl of main thm}.

In Section \ref{sec: lower bounds} we prove Theorem (A.3) by establishing a lower bound on the number of incidences between points and dyadic boxes on the plane, demonstrating that the logarithmic factor is unavoidable  (Proposition \ref{prop: lower bound for boxes}).

In Section \ref{sec: dyadic rects}, we establish Theorem (A.2) by obtaining a stronger bound on the number of incidences with \emph{dyadic} boxes on the plane (Theorem \ref{thm: upper bound for dyadic rects}). We use a different argument relying on a certain partial order specific to the dyadic case to reduce from $\log^4(n)$ given by the general theorem above to $\log(n)$. Up to a constant factor, this implies the same bound for incidences with general boxes when one only counts incidences that are bounded away from the border (Remark \ref{rem: good incidences}). 

Finally, in Section \ref{sec: loc mod}, we prove a general Zarankiewicz bound for definable relations in weakly locally modular geometric first-order structures (Theorem \ref{thm: non-cart lin}), deduce Theorem (D) from it (Corollary \ref{cor: semilin wo log}) and observe how to  recover a real closed field from the failure of Theorem (D) in the $o$-minimal case (Corollary \ref{cor: lin iff loc mod}).

\subsection*{Acknowledgements}
We thank the referees for their very helpful suggestions on improving the paper. Artem Chernikov was partially supported by the NSF CAREER grant DMS-1651321 and by a Simons Fellowship. He is grateful to Adam Sheffer for some very helpful conversations, and to the
American Insitute of Mathematics for additional support.
Sergei Starchenko was supported by the NSF Research Grant DMS-1800806.
Terence Tao was partially supported by NSF grant DMS-1764034 and by a Simons Investigator Award.

\section{Upper bounds}\label{sec: upper bound}
\subsection{Coordinate-wise monotone functions and basic sets}

\gdef\tdt{{\times}\dotsb{\times}}
For an integer  $r\in \NN_{>0}$, by  \emph{an $r$-grid}  (or a \emph{grid} if $r$ is clear from the context)
we mean a cartesian product $B=B_1\tdt B_r$ of some sets $B_1, \ldots, B_r$. As usual, $[r]$ denotes the set $\left\{1, 2, \ldots, r \right\}$.

If $B=B_1\tdt B_r$ is a grid, then by a \emph{sub-grid} we mean 
a subset $C \subseteq B$ of the form $C=C_1\times \dotsb
\times C_r$ for  some $C_i \subseteq  B_i$.

Let $B$ be an $r$-grid, $S$ an arbitrary set and $f: B \to S$ a function.
For $i \in [r]$, set  $$B^{i} = B_1 \times \cdots B_{i-1} \times B_{i+1}\times\cdots \times B_r,$$ and let $\pi_i: B \to B_i$ and $\pi^i: B \to B^i$ be the projection maps.

For $a \in B^i $ and $b \in B_i$, we write $a \oplus_i b$ for the element $c \in B$ with $\pi^i(c) = a$ and $\pi_i(c) = b$. In particular, when $i = r$, $a \oplus_r b = (a, b)$.

\begin{defn}\label{def: coord-wise mon funct}
Let $B$ be an $r$-grid and $(S,<)$ a linearly ordered set. A function $f\colon B\to S$ is
\emph{coordinate-wise monotone} if for any $i\in [r]$, $a,a'\in B^i$
and $b,b'\in B_i$ we have
\[ f(a \oplus_i b)\leq f(a \oplus_i b') \Longleftrightarrow f(a'\oplus_i b)\leq f({a'}\oplus_i b'). \]
\end{defn}

\begin{rem}
	Let $B =B_1\tdt B_r$ be an $r$-grid and $\Gamma$ an ordered abelian group.
  We say that a function $f\colon
  B\to \Gamma$ is \emph{quasi-linear} if  there exist some functions $f_i\colon B_i\to
  \Gamma$, $i\in [r]$, such that  \[ f(x_1,\dotsc,x_r)=f_1(x_1)+\dotsb +f_r(x_r). \]
  
Then every quasi-linear function is coordinate-wise monotone (as  
  $
  f(a\oplus_i b) \leq f({a} \oplus_i b') \Leftrightarrow f_i(b) \leq f_i(b')$ for any $a \in B^i$).

\end{rem}

\begin{sample}\label{ex: lin f impl mon}

  Suppose that $V$ is an ordered vector space over an ordered division ring $R$,  $d_i \in \mathbb{N}$ for $i \in [r]$, and $f: V^{d_1} \times \cdots \times V^{d_r} \to V $ is a linear function. Then $f$ is obviously quasi-linear, hence coordinate-wise monotone.
\end{sample}

\begin{rem} Let $B$ be a grid and $C \subseteq B$ a sub-grid.
  If $f\colon B\to S$ is a coordinate-wise monotone function then the
  restriction $f{\restriction C}$ is a coordinate-wise monotone function on $C$.
\end{rem}

\begin{defn}
Let $B$ be an $r$-grid. 
A subset $X\subseteq B$ is
\emph{a basic set} if there exists a linearly ordered set $(S,<)$, a coordinate-wise monotone function
$f\colon B\to S$ and $l\in S$ such that $X= \left\{ b\in B \colon f(b) <
l\right\}$.
\end{defn}

\begin{rem}
If $r=1$, then every subset of $B=B_1$ is basic.	
\end{rem}

\begin{rem}\label{rem : leq is basic}
  If $X\subseteq B$  is given by  $X= \left\{ b\in B \colon f(b) \leq
l\right\}$ for
some coordinate-wise monotone function
$f\colon B\to S$, then $X$ is a basic set as well. Indeed, we can just
add a new element $l'$ to $S$ so that it is a successor of $l$, then
$X=\left\{b \in B:  f(b)< l'\right\}$.

Similarly, the sets $\left\{ b\in B \colon f(b) >
l\right\}, \left\{ b\in B \colon f(b) \geq
l\right\}$ are basic, by inverting the order on $S$.
\end{rem}

We have the following ``coordinate-splitting'' presentation for basic sets.

\begin{prop}\label{prop:coord}
Let $B=B_1\tdt B_r$ be an $r$-grid and $X\subseteq B$ a basic set. Then there is a linearly ordered
set $(S,<)$, a coordinate-wise monotone function
$f^r\colon B^r \to S$ and a function $f_r\colon B_r\to S$ such that
$X=\left\{ b^r \oplus_{r} b_r \colon f^r(b^r) < f_r(b_r) \right\}$.  
\end{prop}

\begin{rem}
\label{rem:coord}
The converse of this proposition is also true:  an arbitrary linear order $(S,<)$ can be realized as a subset of some ordered abelian group $(G, +, <)$ with the induced ordering (we can take $G := \mathbb{Q}$ when $S$ is at most countable); then define $f: B \to S$ by setting \[ f(b^r \oplus_r b_r) := f^r(b^r) - f_r(b_r), \mbox{ and } l := 0. \]

\end{rem}

\begin{proof}[Proof of Proposition \ref{prop:coord}]
	
Assume that we are given a coordinate-wise monotone function
$f\colon B\to S$ and $l\in S$ with $X= \left\{ b\in B \colon f(b) <
l\right\}$.

For  $i\in [r]$, let $\leq_i$ be the  pre-order on $B_i$ induced by $f$,
namely for $b,b'\in B_i$ we set $b\leq_i b'$ if and only if for some
(equivalently, any) $a\in B^i$ we have $f(a \oplus_i b)\leq f(a \oplus_i b')$.

Quotienting $B_i$ by the equivalence relation corresponding to the pre-order $\leq_i$ if needed, we
may assume that each $\leq_i$ is actually a linear order.

Let $<^r$ be the partial order on $B^r$ with $(b_1,\dotsc,b_{r-1}) <^r   (b'_1,\dotsc,b'_{r-1})$ if and only if
\[ (b_1,\dotsc,b_{r-1}) \neq    (b'_1,\dotsc,b'_{r-1})  \text{ and }
  b_j\leq_j b'_j \text{ for all } j\in[r-1].   \]

Let $T := B^r \dot\cup B_r$, where $\dot\cup$ denotes the disjoint union.
Clearly $<^r$ is a \emph{strict} partial order on $T$, i.e.~a transitive and anti-symmetric (hence irreflexive) relation.

For any $b^r\in B^r$ and $b_r\in B_r$ we define 
\[ b^r\triangleleft b_r \text{ if } f(b^r \oplus_r b_r)  <l 
\text{, and }
b_r\triangleleft b^r  \text{ otherwise}. \]

\begin{claim}
  \label{claim:triangle}
  Let  $a_1,a_2\in B^r$, and $b_1,b_2\in B_r$.
  \begin{enumerate}
  \item If $a_1\triangleleft b_1 \triangleleft a_2 \triangleleft b_2$, then $b_2 <_r b_1$ and $a_1\triangleleft b_2$.
\item If $b_1\triangleleft a_1 \triangleleft b_2 \triangleleft a_2$, then $b_2 <_r b_1$ and $b_1\triangleleft a_2$.  
  \end{enumerate}
\end{claim}
\begin{proof}
$(1)$.   We have $f(a_2 \oplus_r b_1) \geq  l$ and $f(a_2 \oplus_{r} b_2) < l$, hence $b_2 <_r
  b_1$.

Since $f(a_1 \oplus_r b_1)<l$ and $b_2 <_r
  b_1$ we also have $f(a_1 \oplus_r b_2)<l$.

  $(2)$ is similar.
\end{proof}

Let $\triangleleft^t$ be the transitive closure of $\triangleleft$. 
It follows from the above claim that $\triangleleft^t=\triangleleft\cup\triangleleft{\circ}\triangleleft$.
More explicitly, for $b_1,b_2 \in B_r$, $b_1 \triangleleft^t b_2$ if $b_2 <_r b_1$, and for $a_1,a_2 \in B^r$, $a_1 \triangleleft^t a_2$ if $f(a_1 \oplus b) < l < f(a_2 \oplus b)$ for some $b \in B_r$.
It is not hard to see then that $\triangleleft^t$ is
anti-symmetric, hence it is a strict partial order on $T$.

\begin{claim} The union $<^r \cup \triangleleft^t$ is a strict partial order
  on $T$.
  \end{claim}
  \begin{proof}
    We first show transitivity. Note that $<^r$ and $\triangleleft^t$ are both transitive, so it suffices to show for $x, y, z \in T$ that if either $x <^r y \triangleleft^t z$ or $x \triangleleft^t y <^r z$, then $x \triangleleft^t z$. Furthermore, since $\triangleleft^t=\triangleleft\cup\triangleleft{\circ}\triangleleft$, we may restrict our attention to the following cases.    If $a_1 <^r a_2\triangleleft b$ with  $a_1,a_2\in B^r$
    and $b\in B_r$, then $f(a_1 \oplus_r b)<f(a_2 \oplus_r b)<l$, and so
    $a_1\triangleleft b$.
    If $b\triangleleft a_1 <^r a_2$ with $a_1,a_2\in B^r$
    and $b\in B_r$, then $f(a_2 \oplus_r
    b)>f(a_1 \oplus_r b)\geq l$, and so $b\triangleleft a_2$. 

    To check anti-symmetry, assume
    $a_1 <^r a_2$ and $a_2 \triangleleft^t a_1$. Since $a_1,a_2\in
    B^r$ we have $a_2\triangleleft b \triangleleft a_1$ for some $b\in
    B_r$.  We have $f(a_1 \oplus_r b)\geq l > f(a_2 \oplus_r b)$, contradicting
    $a_1<^r a_2$.
  \end{proof}

Finally, let $\prec$ be an arbitrary linear order on $T=B^r\dot\cup B_r$ extending $<^r \cup
\triangleleft^t$.
Since $\prec$ extends $\triangleleft$, for $a\in B^r$
and $b\in B_r$ we have $(a,b)\in X$ if and only if $a\prec
b$.

We take $f^r\colon B^r\to T$ and $f_r\colon B_r\to T$ to be the identity
maps.  Since $\prec$ extends $<^r$, the map $f^r$ is  coordinate-wise monotone. 
\end{proof}

\subsection{Main theorem}

\begin{defn} Let $B=B_1\tdt B_r$ be an $r$-grid.
  \begin{enumerate}
\item Given $s \in \mathbb{N}$, we say that a set $X\subseteq B$ has  \emph{grid-complexity $s$} (in $B$) if $X$ is the intersection of $B$ with at most $s$ basic subsets of $B$. 

\noindent We say that $X$ has \emph{finite grid-complexity} if it has grid-complexity $s$ for some $s \in \mathbb{N}$.
\item For integers $k_1,\dotsc k_r$ we say that $X\subseteq
  B$ is \emph{$K_{k_1,\dotsc,k_r}$-free} is $X$ does not contain a
  sub-grid $C_1\times\dotsb \times C_r\subseteq S$ with $|C_i|=k_i$. 
\end{enumerate}
\end{defn}

In particular, $B$ itself is the only subset of $B$ of grid-complexity $0$.

\begin{sample}\label{ex: compl of semilin sets}

	Suppose that $V$ is an ordered vector space over an ordered division ring, $d = d_1 + \ldots + d_r \in \mathbb{N}$ and 
	\[
X = \left\{ \bar{x} \in V^{d}:f_{1}\left(\bar{x}\right) \leq 0, \ldots, f_{p}\left(\bar{x}\right) \leq 0,f_{p+1}\left(\bar{x}\right)<0,\ldots,f_{s}\left(\bar{x}\right)<0\right\} \mbox{,}
\]
for some linear functions $f_i: V^d \to V, i \in [s]$. Then each $f_i$ is coordinate-wise monotone (Example \ref{ex: lin f impl mon}), hence each of the sets 
$$\left\{\bar{x} \in V^d : f_i(\bar{x}) <0 \right\}, \left\{\bar{x} \in V^d : f_i(\bar{x}) \leq 0 \right\}$$
is a basic  subset of the grid $V^{d_1} \times \ldots \times V^{d_r}$ (the latter by Remark \ref{rem : leq is basic}), and $X \subseteq V^{d_1} \times \ldots \times V^{d_r}$ as an intersection of these $s$ basic sets has grid-complexity $s$.
\end{sample}

\begin{rem} 
\begin{enumerate}
	\item Let $B$ be an $r$-grid and $A\subseteq B$ a subset of $B$ of grid-complexity $s$. If $C \subseteq B$ is
  a sub-grid containing $A$, then $A$ is also a subset of $C$
  of grid-complexity $s$.
  \item  In particular, if $A\subseteq B$ is a subset of grid-complexity
  $s$, then $A$ is a subset of grid-complexity $ s$ of the grid
  $A_1\tdt A_r$, where $A_i :=\pi_i(A)$ is the projection of $A$ on $B_i$
  (it is the smallest sub-grid of $B$ containing $A$).
\end{enumerate} 
\end{rem}

\begin{defn}
	Let $B=B_1\tdt B_r$ be a finite $r$-grid and 
$n_i :=|B_i|$.  For $j\in \{0,\dotsc r\}$, we will denote by $\delta_j^r(B)$
the integer
\[  \delta_j^r(B) := \sum_{ i_1<i_2<\dotsb< i_j \in  [r]} n_{i_1} \cdot n_{i_2} \cdot \ldots  \cdot n_{i_j}.\]  
\end{defn}
\begin{sample}
We have $\delta^r_0(B)=1$,  $\delta^r_1(B)=n_1+\dotsb+n_r$,
$\delta_r^r(B)=n_1n_2\dotsb n_r$.	
\end{sample}

We can now state the main theorem.
\begin{thm}
  \label{thm:main}
For every integers $r\geq 2,  s\geq 0, k\geq 2$  there are
$\alpha=\alpha(r,s,k)\in \RR$ and $\beta=\beta(r,s)\in \NN$ such
that: for any finite $r$-grid $B$ and  $K_{k,\dotsc,k}$-free
 subset  $A \subseteq B$ of grid-complexity $s$ we have
\[ |A| \leq \alpha \delta^r_{r-1}(B) \log^\beta \left( \delta^r_{r-1}(B)+1 \right). \] 

\noindent Moreover, we can take $\beta(r,s) := s(2^{r-1}-1)$.
\end{thm}

\begin{rem}
Inspecting the proof, it can be verified that the dependence of $\alpha$ on $k$ is at most linear.
\end{rem}

\begin{rem}
  We use $\log^\beta \left( \delta^r_{r-1}(B)+1 \right)$ instead of $\log^\beta \left(
  \delta^r_{r-1}(B) \right)$  to include the case $\delta^r_{r-1}(B)\leq 1$.
\end{rem}

\begin{rem}\label{rem: fin union constant}
	If in Theorem \ref{thm:main} $A$ is only assumed to be a union of at most $t$ sets of grid-complexity $s$, then the same bound holds with $\alpha' := t \cdot \alpha$ (if $A = \bigcup_{i \in [t]} A_i$ is $K_{k,\dotsc,k}$-free, then each $A_i$ is also $K_{k,\dotsc,k}$-free, so we can apply Theorem \ref{thm:main} to each of the $A_i$'s and bound $|A|$ by the sum of their bounds).
\end{rem}

\begin{defn}
  \label{defn:delta}
Let $B=B_1\tdt B_r$ be a grid. We
extend the definition of $\delta^r_j$ to arbitrary finite subsets of $B$ as
follows. Let $A\subseteq B$ be a finite subset, and let $A_i :=\pi_i(A)$,
$i\in[r]$, be the projections of $A$. We define
$\delta^r_j(A) :=\delta^r_j(A_1\tdt A_r)$.
\end{defn}

If $B$ is a finite $r$-grid and $A\subseteq B$, then obviously
$\delta^r_j(A)\leq  \delta^r_j(B)$. Thus Theorem~\ref{thm:main} is
equivalent to the following.

\begin{prop}\label{thm:main2}
  For every integers $r\geq 2,  s\geq 0, k\geq 2$  there are
$\alpha=\alpha(r,s,k)\in \RR$ and $\beta=s(2^{r-1}-1) \in \NN$ such
that for any $r$-grid $B$ and  $K_{k,\dotsc,k}$-free finite subset  $A \subseteq B$ of grid-complexity $\leq s$ we have
\[ |A| \leq \alpha \delta^r_{r-1}(A) \log^\beta(\delta^r_{r-1}(A)+1). \] 
\end{prop}

\begin{defn}

For $r\geq 1,  s\geq 0, k\geq 2$ and $n\in \NN$, let $F_{r,k}(s,n)$ be
the maximal size of a $K_{k,\dotsc,k}$-free subset $A$
of grid-complexity $s$ of some $r$-grid $B$  with $\delta_{r-1}^r(B)\leq n$.
\end{defn}

Then Proposition \ref{thm:main2} can be restated as follows.
\begin{prop}\label{thm:main1}
  For every integers $r\geq 2,  s\geq 0, k\geq 2$  there are
$\alpha=\alpha(r,s,k)\in \RR$ and $\beta=\beta(r,s)\in \NN$ such
that
\[ F_{r,k}(s,n) \leq \alpha n \log^\beta( n+1).\]
\end{prop}
\begin{rem}\label{rem:zero}
  Notice that $F_{r,k}(s,0)=0$.
\end{rem}

In the rest of the section we prove Proposition~\ref{thm:main1} by induction on $r$, where for each $r$ it is proved by induction on $s$. We will use the following simple recurrence bound.
\begin{fact}\label{fact:fact}
  Let $\mu\colon \NN \to \NN$ be a function satisfying $\mu(0)=0$ and
  $\mu(n)\leq 2\mu(\lfloor
n/2\rfloor) + \alpha n \log^\beta(n+1))$ for some $\alpha\in \RR$ and
$\beta\in \NN$. Then $\mu(n)\leq \alpha' n  \log ^{\beta+1} (n+1)$ for
some $\alpha'=\alpha'(\alpha,\beta)\in \RR$. 
\end{fact}


\subsection{The base  case $r=2$}
\label{sec:case-r=2}

Let $B=B_1{\times} B_2$ be a finite grid and $A\subseteq B$ a subset of grid-complexity $s$. We will proceed by
induction on $s$.

If $s=0$ then $A=B_1\times
B_2$. If $A$ is $K_{k,k}$-free then one of the sets $B_1, B_2$ must have size
at most $k$. Hence $|A| \leq k(|B_1|+|B_2|)=k\delta^2_1(B)$.

Thus
\[ F_{2,k}(0,n) \leq k n.\]

\begin{rem}\label{rem: s=0}
	The same argument shows that $F_{r,k}(0,n) \leq k n$ for all $r \geq 2$.
\end{rem}

Assume now that the theorem is proved for $r=2$ and all $s'<s$.
Let $n_1 :=|B_1|$, $n_2 :=|B_2|$ and
$n :=\delta^2_1(B)=n_1+n_2$. 

We choose basic sets $X_1,\dotsc X_s \subseteq B$ such
that $A=B \cap \bigcap_{j \in [s]} X_j$.

By Proposition \ref{prop:coord}, we can choose a finite linear order $(S,<)$ and functions  $f_1\colon B_1\to S$ and $f_2\colon B_2\to S$
so that 
\[  X_s=\left\{ (x_1,x_2) \in B_1\times B_2 \colon 
  f_1(x_1)< f_2(x_2)\right\}.\]

For $l\in S$,  $i\in\{1,2\}$ and $\square \in \{ <,=,>, \leq, \geq\}$, let

\[   B_i^{\square l} = \left\{ b\in B_i \colon f_i(b) \square  l \right\}. \]

We choose $h\in S$ such that
\[ |B_1^{<h}|+|B_2^{<h}| \leq n/2 \text{ and } |B_1^{>h}|+|B_2^{>h}| \leq n/2. \]
For example we can take $h$ to be the minimal element in $f_1(B_1)\cup
f_2(B_2)$ with $ |B_1^{\leq h}|+|B_2^{ \leq h}| \geq n/2$.
Then
\begin{gather*} X_s =
  \left[(B_1^{<h}\times B_2^{<h}) \cap X_s \right]
  \cup
  \left[ (B_1^{>h}\times B_2^{>h} ) \cap X_s \right] \\
\cup 
(B_1^{<h}\times B_2^{\geq h} )
\cup
(B_1^{=h}\times B_2^{>h}).
\end{gather*}

Hence we conclude 
$$F_{2,k}(s,n)\leq 2 F_{2,k}(s,\lfloor
n/2\rfloor)+2F_{2,k}(s-1,n).$$

Applying induction hypothesis on $s$, and using Fact~\ref{fact:fact}
and Remark~\ref{rem:zero} we obtain  $F_{2,k}(s,n)\leq \alpha n (\log
n)^\beta$ for some $\alpha=\alpha(s,k)\in \RR$ and $\beta=\beta(s)\in \NN$.

This finishes the base case $r=2$.

\subsection{Induction step}
\label{sec:induction-step-1}
We fix $r \in \mathbb{N}_{\geq 3}$ and assume that Proposition~\ref{thm:main1} holds for all pairs $(r',s)$ with $r'<r$ and $s \in \mathbb{N}$.

\begin{defn}
	Let $B=B_1\tdt B_r$ be a finite $r$-grid.
	\begin{enumerate}
		\item For integers $t, u\in \NN$, we say that a subset $A\subseteq B$ is \emph{of split grid-complexity $(t, u)$} if there are basic sets $X_1,\dotsc, X_{u}
\subseteq B$, a subset $A^r\subseteq B_1\tdt B_{r-1}$ of
grid-complexity $t$, and a subset $A_r\subseteq B_{r}$   such that $A=(A^r\times
A_r)\cap \bigcap_{i \in [u]} X_i$.

\item 
For $t, u \geq 0, k\geq 2$ and $n\in \NN$, let $G_{k}(t,u,n)$ be
the maximal size of a $K_{k,\dotsc,k}$-free subset $A$ of an
$r$-grid $B$  of split grid-complexity $(t,u)$ with $\delta_{r-1}^r(B)\leq n$.
	\end{enumerate}
\end{defn}

\begin{rem}
\begin{enumerate}
   \item Note that $A_r$ has grid-complexity at most $1$, which is the reason we do not include a parameter for the grid-complexity of $A_r$ in the split grid-complexity of $A$. 
   
    \item If $A\subseteq B$ is of grid-complexity $s$, then it is of split grid-complexity $(0,s)$. 
    \item If $A\subseteq B$ is of split grid-complexity $(t, u)$, then it is of grid-complexity $t + u$.
\end{enumerate}
\end{rem}

For the rest of the proof, we abuse notation slightly and refer to the ``split grid-complexity'' of a set as the ``grid-complexity''. To complete the induction step we will prove the following Proposition.
\begin{prop}\label{prop:main}
  For any integers $t,u\geq 0, k\geq 2, r \geq 3$  there are
$\alpha' = \alpha'(r,k,t,u)\in \RR$ and $\beta' = \beta'(r,k,t,u)\in \NN$ such
that
\[ G_{k}(t,u,n) \leq \alpha' n \log^{\beta'}(n+1).\]
\end{prop}

We will use the following notations throughout the section:
\begin{itemize}
\item  $B=B_1\tdt B_r$ is a finite grid with $n=\delta^r_{r-1}(B)$;
\item $A\subseteq B$ is  a subset of grid-complexity $(t,u)$;
\item $B^r$ is the $(r-1)$-grid $B^r :=B_1\tdt B_{r-1}$;
\item  $A^r \subseteq B^r$ is a subset of grid-complexity $t$, $A_r \subseteq B_r$, 
and $X_1,\dotsc X_{u} \subseteq B$ are basic subsets such that 
$A= (A^r{\times} A_r)\cap\bigcap_{i \in [u]} X_i$.  
\end{itemize}

\medskip 
We proceed by induction on $u$.

\medskip

\noindent \emph{The base case $u=0$ of Proposition \ref{prop:main}.}

In this case  $A=A^r\times A_r$. 
If $A$ is $K_{k,\dotsc,k}$-free then either $A^r$ is
  $K_{k,\dotsc ,k}$-free or $|A_r|<k$. 

In the first case, by induction hypothesis on $r$, there are
$\alpha=\alpha(r-1, t,k)$ and $\beta=\beta(r-1, t)$ such that
$|A^r| \leq \alpha \delta^{r-1}_{r-2}(B^r)\log^\beta(
\delta^{r-1}_{r-2}(B^r)+1)$.  In the second case we have
$|A|\leq |B^r|k=\delta^{r-1}_{r-1}(B^r)k$.  

Since $n=\delta^r_{r-1}(B)=\delta^{r-1}_{r-1}(B^r)+\delta^{r-1}_{r-2}(B^r)
|B_r|$, the conclusion of the proposition follows with $\alpha' := \alpha, \beta' := \beta$.

\medskip
\noindent \emph{Induction step of Proposition \ref{prop:main}.} We assume now that the proposition holds for all pairs $(t,u')$ with $u'<u$ and $t \in \mathbb{N}$.

Given a tuple $x = (x_1, \ldots, x_r) \in B$, we let $x^r := (x_1, \ldots, x_{r-1})$. By Proposition \ref{prop:coord}, we can choose a finite linear order $(S,<)$, a coordinate-wise monotone function  $f^r\colon B^r\to S$ and a
function  $f_r\colon B_r\to S$
so that
\[  X_u=\left\{ x^r \oplus_r x_r \in B^r\times B_r \colon 
  f^r(x^r)< f_r(x_r)\right\}.\]
Moreover, by Remark~\ref{rem:coord}, we may assume without loss of generality that the coordinate-wise monotone function defining $X_u$ is given by \[ f(x^r \oplus_r x_r)=f^r(x^r) - f_r(x_r). \]

\begin{defn}
Given an arbitrary set $C^r \subseteq B^r$,
we say that a set $H^r \subseteq C^r$ is  an \emph{$f^r$-strip in $C^r$} if
\[ H^r=\left\{ x^r \in C^r  \colon l_1 \triangleleft_1 f^r(x^r) \triangleleft_2 l_2\right\}\]
for some $l_1,l_2\in S$, $\triangleleft_1, \triangleleft_2\in \{
<,\leq\}$.	Likewise, given an arbitrary set $C_r \subseteq B_r$, we say that $H_r \subseteq  C_r$ is an \emph{$f_r$-strip in $C_r$} if 
\[ H_r=\left\{ x_r \in C_r  \colon l_1 \triangleleft_1 f_r(x_r) \triangleleft_2 l_2\right\}\] for some $l_1,l_2\in S$, $\triangleleft_1, \triangleleft_2\in \{
<,\leq\}$. If $C^r = A^r$ or $C_r = A_r$, we simply say \emph{an $f^r$-strip} or \emph{$f_r$-strip}, respectively.
\end{defn}

\begin{rem}\label{rem: props of f-strips}
Note the following:
\begin{enumerate}
	\item $A^r$ is an $f^r$-strip, and $A_r$ is an $f_r$-strip;
	\item  every $f^r$-strip is a subset of the $(r-1)$-grid $B^r$ of grid-complexity $t+2$ (using Remark \ref{rem : leq is basic});
	\item the intersection of any two
$f^r$-strips is an $f^r$-strip;  the same conclusion holds for $f_r$-strips.
\end{enumerate}
\end{rem}

\begin{defn}
\begin{enumerate}
	\item We say that a subset $H\subseteq B$ is an \emph{$f$-grid}
if $H=H^r\times H_r$, where $H^r\subseteq B^r$ is an $f^r$-strip in $B^r$
and $H_r \subseteq B_r$ is an $f_r$-strip in $B_r$.
\item If $H=H^r\times H_r$ is an $f$-grid, we set
\[ \Delta(H) :=|H^r| + \delta^{r-1}_{r-2}(H^r)|H_r|
  \text{ (see Definition ~\ref{defn:delta} for } \delta^{r-1}_{r-2}). \]
  Note that if $H$ is a sub-grid of $B$, then $\Delta(H)=\delta^r_{r-1}(H)$. 
 
 \item For an $f$-grid $H$, we will denote by $A_H$ the set $A\cap H$.

\end{enumerate}	
\end{defn}

The induction step for Proposition~\ref{prop:main} will follow from the
following proposition.

\begin{prop}\label{prop:main1}
  For every integer $k\geq 2, r \geq 3$ there are $\alpha' = \alpha'(r,k,t,u)\in \RR$ and
  $\beta'=\beta'(r,t,u) \in \NN$ such that, for any $f$-grid $H$, if the set $A_H$ is
  $K_{k,\dotsc,k}$-free then
  \[ |A_H| \leq \alpha' \Delta(H)\log^{\beta'}(\Delta(H)+1). \] 
\end{prop}
We should stress that in the above proposition $\alpha'$ and $\beta'$ do
not depend on $f^r, f_r$, $B$, $A^r$, and $A_r$ but they may depend on our fixed $t$
and $u$. 

Given Proposition \ref{prop:main1}, we can apply it to the $f$-grid $H := A^r\times A_r$ (so $A_H = A$) and get 
 $$|A|\leq  \alpha' \Delta(H)\log^{\beta'}(\Delta(H)+1).$$
It is easy to see that
$\Delta(A^r\times A_r)\leq \delta^r_{r-1}(B)$, hence 
Proposition~\ref{prop:main} follows with the same $\alpha'$ and $\beta'$.

\medskip
We proceed with  the proof of Proposition~\ref{prop:main1}

\begin{proof}[Proof of Proposition \ref{prop:main1}]Fix $m\in \NN$, and let $L(m)$ be the maximal size of a $K_{k,\dotsc,k}$-free
set $A_H$ among all $f$-grids $H \subseteq B$ with $\Delta(H)\leq m$.
We need to show that for some $\alpha'=\alpha'(k)\in \RR$ and $\beta' \in
\NN$
we have
\[L(m)\leq \alpha' m \log^{\beta'}(m+1).\]

Let $H=H^r\times H_r$ be an $f$-grid with $\Delta(H)\leq m$.

For $l\in S$  and $\square \in \{ <,=,>, \leq, \geq\}$, let

\[   H^{r,\square l} := \left\{ x^r \in H^r \colon f^r(x^r) \square  l \right\} \]
and
\[   H_r^{\square l} := \left\{ x_r \in H_r \colon f_r(x_r) \square  l \right\}. \]

Note that for every $l \in S$, $H^{r,\square l}$ is an $f^r$-strip in $H^r$, $H_r^{\square l}$ is an $f_r$-strip in $H_r$, and their product is an $f$-grid.

\begin{claim}
  \label{claim:h}
  There is $h\in S$ such that 
  \[ \Delta(H^{r,<h}\times H_r^{<h}) \leq m/2 \text{ and }
\Delta(H^{r,>h}\times H_r^{>h}) \leq m/2.\]
\end{claim}
\begin{proof}[Proof of Claim]
  Let $\delta :=\delta^{r-1}_{r-2}(H^r)$.

  Let $h$ be the minimal element in $f^r(H^r)\cup
  f_r(H_r)$ with \[ |H^{r,\leq h }|+\delta |H_r^{\leq h}| \geq m/2.\]
  Then
  $ |H^{r,< h }|+\delta |H_r^{< h}| \leq m/2$ and 
  $ |H^{r, > h }|+\delta |H_r^{> h}| \leq m/2$.
  Since $H^{r,< h}, H^{r,> h} \subseteq H^r$, we have
  $\delta^{r-1}_{r-2}(H^{r,< h}), \delta^{r-1}_{r-2}(H^{r,> h})  \leq \delta$.
The claim follows. 
\end{proof}
Let $h$ be as in the claim. It is not hard to see that the following holds:
\begin{gather*}
	\left( H^{r, \leq h} \times H_r^{\geq h} \right)
\cap X_u =  \left( H^{r, < h} \times H_{r}^{\geq h} \right) \cup \left( H^{r, =h} \times H_r^{>h} \right),\\
\left( H^{r, \geq h} \times H_r^{\leq h} \right) \cap X_u = \emptyset.
\end{gather*}
It follows that
\begin{gather*} A_H \cap X_u =
  \left[(H^{r,<h}\times H_r^{<h}) \cap X_u \right] 
  \cup
   \left[(H^{r,>h}\times H_r^{>h}) \cap X_u \right] \\
\cup 
(H^{r,<h}\times H_r^{\geq h})
\cup
(H^{r,=h}\times H_r^{>h}).
\end{gather*}

Hence, by the choice of $h$ and using Remark \ref{rem: props of f-strips}(2), 
$$L(m)\leq 2 L(\lfloor
m/2\rfloor)+ 2G_{k}(t+2,u-1,m).$$

Applying the induction hypothesis on $u$ and using Fact~\ref{fact:fact}
 we obtain  $L(m)\leq \alpha' m \log^{\beta'}(m+1)$
for some $\alpha'=\alpha'(k)\in \RR$ and $\beta'\in \NN$. 

This finishes the proof of Proposition \ref{prop:main1}, and hence of the
induction step of Proposition \ref{thm:main1}.
\end{proof}

\medskip

Finally, inspecting the proof, we have shown the following:
\begin{enumerate}
	\item $\beta(2,s) \leq s$ for all $s \in \mathbb{N}$; 

 \item $\beta'(r,t, 0) \leq \beta(r-1,t)$ for all $r \geq 3$ and $t \in \mathbb{N}$;

 \item $\beta'(r,t,u) \leq \beta'(r, t+2, u-1) + 1$ for all $r \geq 3, t \geq 0, u \geq 1$. 
\end{enumerate}	

Iterating (3), for every $r \geq 3, s \geq 1$ we have $\beta(r,s) \leq \beta'(r,0,s) \leq \beta'(r, 2s, 0) + s$. Hence, by (2), $\beta(r,s) \leq \beta(r-1, 2s) + s$ for every $r \geq 3$ and $s \geq 1$. Iterating this, we get 
$\beta(r,s) \leq \beta(2, 2^{r-2}s) + s \sum_{i=0}^{r-3}2^i$. Using (1), this implies $\beta(r,s) \leq s \sum_{i=0}^{r-2}2^i  = s(2^{r-1}-1)$ for all $r \geq 3, s \geq 1$. Hence, by Remark \ref{rem: s=0} and (1) again, $\beta(r,s) \leq s(2^{r-1}-1)$ for all $r \geq 2, s \geq 0$.
\subsection{Some applications}\label{sec: appl of main thm}

	We observe several immediate applications of Theorem 
	\ref{thm:main}, starting with the following bound for semilinear hypergraphs.
	
\begin{cor}\label{cor: semilin Zar}

For every $r,s,t,k \in \mathbb{N}, r \geq 2$ there exist some $\alpha=\alpha(r,s,t,k)\in \RR$ and $\beta(r,s) := s(2^{r-1}-1)$ satisfying the following.

For any semilinear $K_{k, \ldots, k}$-free $r$-hypergraph $H = (V_1, \ldots, V_r;E)$ of description complexity $(s,t)$ (see Definition \ref{def: semilin hypergraphs}), taking $V:= \prod_{i \in [r]}V_i$ we have 
\[ |E| \leq \alpha \delta^r_{r-1}(V) \log^\beta \left( \delta^r_{r-1}(V)+1 \right). \]
\end{cor}
\begin{proof}
	By assumption the edge relation $E$ can be defined by a union of $t$ sets, each of which is defined $s$ linear  equalities and inequalities, hence of grid-complexity $\leq s$ (see Example \ref{ex: compl of semilin sets}). The conclusion follows by Theorem \ref{thm:main} and Remark \ref{rem: fin union constant}.
\end{proof}

As a special case with $r=2$, this implies a bound for the following incidence problem.

\begin{cor}\label{cor: inc with polytopes}
	For every $s,k \in \mathbb{N}$ there exists some $\alpha=\alpha(s,k)\in \RR$ satisfying the following. 
	
	Let $d \in \mathbb{N}$ and $H_1, \ldots, H_s \subseteq \mathbb{R}^d$ be finitely many (closed or open) half-spaces in $\mathbb{R}^d$. Let $\mathcal{F}$ be the (infinite) family of all possible polytopes in $\mathbb{R}^d$ cut out by arbitrary translates of $H_1, \ldots, H_s$.

	For any set $P$ of $n_1$ points in $\mathbb{R}^d$ and any set $F$ of $n_2$ polytopes in $\mathcal{F}$, if the incidence graph on $P \times F$ is $K_{k,k}$-free, then it contains at most $\alpha n \log^{s} n$ incidences.
\end{cor}
\begin{proof}
We can write 
$$H_i = \left\{\bar{x} = (x_1, \ldots, x_d) \in \mathbb{R}^d : \sum_{j \in [d]} a_{i,j} x_j \square_i b_i \right\},$$ where $a_{i,j},b_i \in \mathbb{R}$ and $\square_i \in \{ >, \geq \}$ for $i \in [s], j \in [d]$ depending on whether $H_i$ is an open or a closed half-space.

Every polytope $F \in \mathcal{F}$ is of the form $\bigcap_{i \in [s]} (\bar{y}_i + H_i)$ for some $(\bar{y}_1, \ldots, \bar{y}_s) \in \mathbb{R}^{sd}$, where $\bar{y}_i + H_i$ is the translate of $H_i$ by the vector $\bar{y}_i = (y_{i,1}, \ldots, y_{i,d}) \in \mathbb{R}^d$, i.e.~
$$\bar{y}_i + H_i = \left\{ \bar{x} \in \mathbb{R}^d : \sum_{j \in [d]} a_{i,j} x_j + \sum_{j \in [d]}(-a_{i,j})y_j \square_i b_i \right\}.$$

Then the incidence relation between points in $\mathbb{R}^d$ and polytopes in $\mathcal{F}$ can be identified with the semilinear set 
$$\left\{ \left(\bar{x}; (y_{i,j})_{i \in [s], j \in [d]} \right) \in \mathbb{R}^d \times \mathbb{R}^{sd} : \bigwedge_{i \in [s]}  \sum_{j \in [d]} a_{i,j} x_j + \sum_{j \in [d]}(-a_{i,j})y_{i,j} \square_i b_i  \right\}$$
defined by $s$ linear inequalities. 
 The conclusion now follows by Corollary \ref{cor: semilin Zar} with $r=2$.
\end{proof}

 In particular, we get a bound for the original question that motivated this paper.
\begin{cor}\label{cor: inc with boxes}
	Let $\mathcal{F}_d$ be the family of all (closed or open) boxes in $\mathbb{R}^d$. Then for every $k$ there exists some $\alpha = \alpha(d,k)$ satisfying the following.
	
	For any set $P$ of $n_1$ points in $\mathbb{R}^d$ and any set $F$ of $n_2$ boxes in $\mathcal{F}_d$, if the incidence graph on $P \times F$ is $K_{k,k}$-free, then it contains at most $\alpha  n  \log^{2 d} n $ incidences.
\end{cor}
\begin{proof}
	Immediate from Corollary \ref{cor: inc with polytopes}, since we have $2d$ half-spaces in $\mathbb{R}^d$ so that every box in $\mathbb{R}^d$ is cut out by the intersection of their translates.
\end{proof}

\section{Lower bounds}\label{sec: lower bounds}

While we do not know if the bound $\beta(2,s) \leq s$ in Theorem \ref{thm:main} is optimal, in this section we show that at least the logarithmic factor is unavoidable already for the incidence relation between points and dyadic boxes in $\mathbb{R}^2$.

We describe a slightly more general construction first. Fix $d \in \mathbb{N}_{>0}$.
\begin{defn}
	Given finite tuples $\bar{p}=(p_1, \ldots, p_{n}), \bar{q}=(q_1, \ldots ,q_n)$ and $\bar{r}=(r_1, \ldots, r_m)$ with $p_i,q_i,r_i \in \mathbb{R}^d$, say $p_i = (p_{i,1}, \ldots, p_{i,d}), q_i = (q_{i,1}, \ldots, q_{i,d}), r_i = (r_{i,1}, \ldots, r_{i,d})$, we say that \emph{$\bar{p}$ and $\bar{q}$ have the same order-type over $\bar{r}$} if 
	$$p_{i,j} \square p_{i',j'} \iff q_{i,j} \square q_{i',j'} \mbox{ and}$$
	$$p_{i,j} \square r_{k,j'} \iff q_{i,j} \square r_{k,j'}$$	for all $\square \in \{<,>,= \}$, $1 \leq i,i' \leq n, 1 \leq j,j' \leq d$ and $1 \leq k \leq m$.
	\end{defn}
	In other words, the tuples $(p_{i,j} : 1 \leq i \leq n, 1\leq j \leq d)$ and $(q_{i,j} : 1 \leq i \leq n, 1\leq j \leq d)$ have the same quantifier-free type over the set $\{r_{i,j} : 1 \leq i \leq m, 1 \leq j \leq d \}$ in the structure $(\mathbb{R}, <)$.

\begin{rem}\label{rem: same type}

	Assume that $P =\{ p_1, \ldots, p_{n} \} \subseteq \mathbb{R}^d$ is a finite set of points  and $B$ is a finite set of $d$-dimensional open boxes with axis-parallel sides, with $I$ incidences between $P$ and $B$.
	\begin{enumerate}
		\item By perturbing $P$ and $B$ slightly, we may assume that for every $1 \leq j \leq d$, all points in $P$ have pairwise distinct $j$th coordinates $p_{1,j}, \ldots, p_{n,j}$, and none of the points in $P$ belongs to the border of any of the boxes in $B$, while the incidence graph between $P$ and $B$ remains unchanged.
		\item Let $\bar{r}$ be the tuple listing all corners of all boxes in $B$. If $P' = \{p'_1, \ldots,  p'_n\} \subseteq \mathbb{R}^d$ is an arbitrary set of points with the same order-type as $P$ over $\bar{r}$, then the incidence graph on $P \times B$ is isomorphic to the incidence graph on $P' \times B$.

	\end{enumerate}
	
\end{rem}
We have the following lemma for combining point-box incidence configurations in a higher-dimensional space.
\begin{lem}\label{lem: stepping up}
Given any $d,n_1,n_2,n'_1, n'_2, m,m' \in \mathbb{N}_{>0}$, assume that:
\begin{enumerate}
	\item there exists a set of points $P^{d-1} \subseteq \mathbb{R}^{d-1}$ with $|P^{d-1}| = n_1$ and a set of $(d-1)$-dimensional boxes $B^{d-1}$ with $|B^{d-1}| = n_2$, with $m$ incidences between them, and the incidence graph $K_{2,2}$-free;
	\item there exists a set of points $P^d \subseteq \mathbb{R}^d$ with $|P^d| = n'_1$ and a set of $d$-dimensional boxes $B^d$ with $|B^d|=n'_2$, with $m'$ incidences between them and the incidence graph $K_{2,2}$-free.
	\end{enumerate}
	Then there exists a set of points $P \subseteq \mathbb{R}^d$ with $|P| = n_1 n'_1$ and a set of $d$-dimensional boxes $B$ with $|B| = n_1n'_2+n'_1n_2$, so that there are $n_1m' + m n'_1$ incidences between $P$ and $B$ and their incidence graph is still $K_{2,2}$-free.

\end{lem} 

\begin{proof}
	By Remark \ref{rem: same type}(1) we may assume that for every $1 \leq j \leq d$, all points in $P^d$ have pairwise distinct $j$th coordinates, for every $1 \leq j \leq d-1$ all points in $P^{d-1}$ have pairwise distinct $j$th coordinates, and none of the points is on the border of any of the boxes. Write $P^{d-1}$ as $p_1, \ldots, p_{n_1}$. Let $\bar{r}$ be the tuple listing all corners of all boxes in $B^{d-1}$.
	
	Using this, for each $p_i$ we can choose a very small $(d-1)$-dimensional box $\beta_i$ with $p_i \in \beta_{i}$ and such that: for any choice of points $p'_i \in \beta_i, 1 \leq i \leq n_1$, we have that $(p'_1, \ldots, p'_{n_1})$ has the same order-type as $(p_1, \ldots, p_{n_1})$ over $\bar{r}$. In particular, all the $\beta_i$'s are pairwise disjoint, and the incidence graph between $P^{d-1}$ and $B^{d-1}$ is isomorphic to the incidence graph between $(p'_i, \ldots, p'_{n_1})$ and $B^{d-1}$ by Remark \ref{rem: same type}(2).
	
	Contracting and translating while keeping the $d$th coordinate unchanged, for each $1 \leq i \leq n_1$ we can find a copy $(P^{d}_i, B^d_i)$ of the configuration $(P^d, B^d)$ entirely contained in the box $\beta_i \times \mathbb{R}$, that is:
	\begin{itemize}
		\item all points in $P^{d}_i$ and boxes in $ B^d_i$ are contained in $\beta_i \times \mathbb{R}$;
		\item the incidence graph on $(P^{d}_i, B^d_i)$ is isomorphic to the incidence graph on $(P^{d}, B^d)$;
		\item for all $i$, the $d$th coordinate of every point in $P_i^d$ is the same as the $d$th coordinate of the corresponding point in $P^d$.
	\end{itemize}
Let $P := \bigcup_{1 \leq i \leq n_1} P^d_i$ and $B' := \bigcup_{1 \leq i \leq n_1} B^d_i$, then $|P| = n_1n'_1, |B'| = n_1n'_2$ and there are $n_1m'$ incidences between $P$ and $B'$.

Write $P^d$ as $q_1, \ldots, q_{n'_1}$ and $B^{d-1}$ as $c_1, \ldots, c_{n_2}$. As all of the $d$th coordinates of the points in $P^d$ are pairwise disjoint, for each $1 \leq j \leq n'_1$ we can choose a small interval $I_j \subseteq \mathbb{R}$ with $q_{j,d} \in I_j$, and so that all of the intervals $I_j, 1 \leq j \leq n'_1$ are pairwise disjoint. For each $1 \leq j \leq n'_1$ and $c_l \in B^{d-1}$, we consider the $d$-dimensional box $c_{j,l} :=c_l \times I_j$. Let $B_j := \{c_{j,l} : 1 \leq l \leq n_2 \}$. For each $1 \leq i \leq n_1$ and $1 \leq j \leq n'_1$, $(\beta_i \times \mathbb{R}) \cap (\mathbb{R}^{d-1} \times I_j)$ contains exactly one point $q_{i,j}$ (given by the copy of $q_{j}$ in $P_i^d$), and the projection $q'_{i,j}$ of $q_{i,j}$ onto the first $d-1$ coordinates is in $\beta_i$. Hence the incidence graph between $P$ and $B_j$ is isomorphic to the incidence graph between $P^{d-1}$ and $B^{d-1}$ by the choice of the $\beta_i$'s, in particular the number of incidences is $m$. 

Finally, let $B := B' \cup \bigcup_{1 \leq j \leq n'_1} B_j$, then $|B| = n_1n'_2 + n'_1n_2$. Note that $c_{j,l} \cap c_{j',l'} = \emptyset$ for $j \neq j'$ and any $l,l'$, i.e.~no box in $B_j$ intersects any of the boxes in $B_{j'}$ for $j\neq j'$. It is now not hard to check that the incidence graph between $P$ and $B$ is $K_{2,2}$-free  (by construction and the assumptions of $K_{2,2}$-freeness of $(P^d,B^d)$ and $(P^{d-1}, B^{d-1})$), and that there are $n_1m'+mn'_1$ incidences between $P$ and $B$.
\end{proof}

\begin{rem}\label{rem: wlog dyadic}
	It follows from the proof that if all the boxes in $B^{d-1}$ and $B^d$  are dyadic (see Definition \ref{def: dyadic box}), then we can choose the boxes in $B$ to be dyadic as well.
\end{rem}

\begin{prop}\label{prop: lower bound for boxes}
	For any $\ell \in \mathbb{N}$, there exist a set $P$ of $ \ell^{\ell}$ points and a set  $B$ of $\ell^{\ell}$ dyadic boxes in $\mathbb{R}^2$ such that their incidence graph is $K_{2,2}$-free and the number of incidences is $\ell \ell ^{\ell}$.
	
	In particular, substituting $n := \ell^{\ell}$, this shows that the number of incidences grows as $\Omega \left(n \frac{\log n}{\log \log n} \right)$.
\end{prop}
\begin{proof}
Given $d$, assume that there exist $K_{2,2}$-free `point -- dyadic box' configurations satisfying (1) and (2) in Lemma \ref{lem: stepping up} for some parameters $d, n_1, n_2, n'_1, n'_2, m, m'$.
Then, for any $j \in \mathbb{N}$, we can iterate the lemma $j$ times and find a $K_{2,2}$-free `point -- dyadic box' configuration in $\mathbb{R}^d$ with $n_1^j n'_1$ points, $n_1^j n'_2 + j n_1^{j-1} n'_1 n_2$ dyadic boxes (Remark \ref{rem: wlog dyadic}), and $n_1^j m' + j n_1^{j-1} n'_1 m$ incidences.

In particular, let $d = 2$  and let $\ell$ be arbitrary. We can start with $n_1 = \ell , n_2 = 1, m=\ell$ (one dyadic interval containing $n_1$ points in $\mathbb{R}$) and $n'_1=1, n'_2=0, m' = 0$ (one point and zero dyadic boxes in $\mathbb{R}^2$). Taking $j := \ell$, we then find a $K_{2,2}$-free configuration with $\ell^{\ell}$ points, $\ell^\ell$ dyadic boxes and $\ell \ell^{\ell}$ incidences. Hence for $n := k^k$, we have $n$ points, $n$ boxes and $\Omega \left(n \frac{\log n}{\log \log n} \right)$ incidences.
\end{proof}

\begin{rem}
	We remark that the construction in Lemma \ref{lem: stepping up} cannot produce a $K_{2,2}$-free configuration with more than $O \left( n \frac{\log n}{\log \log n} \right)$ incidences in $\mathbb{R}^d$ for any $d$. 
	
	Indeed, using the ``coordinates'' $\left( \log n'_1, \frac{n'_2}{n'_1}, \frac{m'}{n'_1} \right)$ instead of $(n'_1,n'_2,m')$, where the coordinates correspond to the number of points, boxes and incidences respectively, the lemma says that  if $\left( \log n_1, \frac{n_2}{n_1}, \frac{m}{n_1} \right)$ is attainable in $d-1$ dimensions and $\left( \log n'_1, \frac{n'_2}{n'_1}, \frac{m'}{n'_1} \right)$ is attainable in $d$ dimensions, then $\left( \log n'_1 + \log n_1, \frac{n'_2}{n'_1} + \frac{n_2}{n_1}, \frac{m'}{n'_1} + \frac{m}{n_1} \right)$ is attainable in $d$ dimensions. Thus, one adds the vector $\left(\frac{n_2}{n_1}, \frac{m}{n_1} \right)$ to $\left( \frac{n'_2}{n'_1},\frac{m'}{n'_1} \right)$. We want to maximize the second coordinate of this vector while keeping the first coordinate below $1$, and the optimal way to do it essentially is to add $n_1$ times the vector $\left(\frac{1}{n_1},1 \right)$, which increases $\log n'_1$ by $n_1 \log n_1$ and gives the $\frac{\log n}{\log \log n}$ lower bound. 
	
	We thus ask whether in the `point-box' incidence bound in $\mathbb{R}^d$ the power of $\log n$ has to grow with the dimension $d$ (see Problem \ref{prob: log power}).
\end{rem}

\section{Dyadic rectangles}\label{sec: dyadic rects}

In this section we strengthen the bound on the number of incidences with rectangles on the plane with axis-parallel sides given by Corollary \ref{cor: inc with boxes}, i.e.,~$O_{k} \left( n \log^{4} n \right)$, in the special case of \emph{dyadic} rectangles, using a different argument (which relies on a certain partial order specific to the dyadic case).

\subsection{Locally $d$-linear orders}

Throughout this section, let $(P, \leq)$ be a partially ordered set of size at most $n_1$, and let $L$ be a collection of subsets of $P$ (possibly with repetitions) of size at most $n_2$. As before, we let $n = n_1 + n_2$.

\begin{defn}
We say that a set $S \subseteq P$ is \emph{$d$-linear} if it contains no antichains of size greater than $d$, and $(P, \leq)$ is \emph{locally $d$-linear} if any interval $[a,b] = \{ x \in P : a \leq x \leq b \}$ is $d$-linear.
\end{defn}

Note that $d$-linearity is preserved under removing points from $P$.

\begin{defn}
The collection $L$ is said to be a \emph{$K_{k,k}$-free arrangement} if for any $a_1 \neq  \ldots \neq  a_k \in P$, there are at most $k-1$ sets from $L$ containing all of them simultaneously.
\end{defn}

Observe that if one removes any number of points from $P$, or removes any number of sets from $L$, one still obtains a $K_{k,k}$-free  arrangement. We now state the main theorem of this section.

\begin{thm}\label{main}  Suppose $(P,<)$ is locally $d$-linear, and $L$ is a $K_{k,k}$-free arrangement of $d$-linear subsets of $P$.  Then
$$\sum_{\ell \in L} |\ell| = O_{d,k}\left( n \frac{\log(100+n_1)}{\log \log(100+n_1)} \right)$$.
\end{thm}

To prove Theorem~\ref{main}, we first need some definitions and a lemma. If $x \in P$, define a \emph{parent} of $x$ to be an element $y \in P$ with $y>x$ and no element between $x$ and $y$, and similarly define a \emph{child} of $x$ to be an element $z \in P$ with $z<x$ and no element between $z$ and $x$. We say that $z$ is a \emph{strict $t$-descendant} of $x$ if there are some elements $z_0 = x > z_1 > \ldots > z_{t} = z$ such that $z_{i+1}$ is a child of $z_i$, and that $z$ is a \emph{$t$-descendant} of $x$ if it is a strict $s$-descendant for some $0 \leq s \leq t$. 

\begin{lem}  Fix $d,k \in \mathbb{N}$. Let $L$ be a $K_{k,k}$-free arrangement of $d$-linear subsets of $P$, and let $m> 0$.  Let $P'$ denote the set of all elements in $P$ which have a $(k-1)$-descendant with more than $m$ children.  Then
$$\sum_{\ell \in L} |\ell| \leq \sum_{\ell \in L} |\ell \cap P'| + d(k-1)|L| + (k-1) m^{k-1} (|P| - |P'|).$$
\end{lem}

\begin{proof}  Let $P'' := P \backslash P'$ denote the set of elements $x \in P$ such that every $(k-1)$-descendant of $x$ has at most $m$ children.  Then we can rearrange the desired inequality as
$$ \sum_{\ell \in L} |\ell \cap P''| \leq d(k-1)|L| + (k-1)m^{k-1} |P''|.$$
The quantity $\sum_{\ell \in L} |\ell \cap P''|$ is counting incidences $(x,\ell)$ where $\ell \in L$ and $x \in P'' \cap \ell$.  

Given $\ell \in L$, call a point $x \in \ell$ low if $x$ has no descending chain of length $k-1$ under it in $\ell$. Every $\ell$ can contain at most $d(k-1)$ low points. Indeed, as $\ell$ is $d$-linear, it has at most $d$ minimal elements. Removing them, we obtain a $d$-linear set $\ell_1 \subseteq \ell$ such that every point in it contains an element under it in $\ell$, and $\ell_1$ itself has at most $d$ minimal elements. Remove them to obtain a $d$-linear set $\ell_2 \subseteq \ell_1$ such that each point in it contains a descending chain of length $2$ under it in $\ell$, etc.

Hence each $\ell \in L$ contributes at most $d(k-1)$ incidences with its low points, giving a total contribution of at most $d(k-1)|L|$ to the sum.  If $x$ is not a low point on $\ell$, then there are some $z_1 < \ldots < z_{k-1} < x$ in $\ell$, with each one a child of the next one.  As $L$ is a $K_{k,k}$-free arrangement, among the sets $\ell \in L$ there are at most $k-1$ containing all these points.  By definition of $P''$, for each $x \in P''$ there are at most $m^{k-1}$ choices for such tuples $(z_1, \ldots, z_{k-1})$. Hence $x$ is incident to at most $(k-1)m^{k-1}$ sets $\ell \in L$ for which it is not low, and the total number of contributions of incidences in this case is at most $(k-1) m^{k-1} |P''|$, so the claim follows.
\end{proof}

Now we prove Theorem \ref{main}.  Let $t$ be a natural number to be chosen later, and $m>0$ be another parameter to be chosen later.  Define the subsets
$$ P = P_0 \supset P_1 \supset \dots \supset P_t$$
of $P$ by defining $P_0 := P$, and for each $i=0,\dots,t-1$, defining $P_{i+1}$ to be the set of points in $P_i$ that have a $(k-1)$-descendant with more than $m$ children in $(P_i, <)$.   By the above lemma, we have
$$\sum_{\ell \in L} |\ell \cap P_i| \leq \sum_{\ell \in L} |\ell \cap P_{i+1}| + d(k-1)|L| + (k-1)m^{k-1} (|P_i| - |P_{i+1}|)$$
for all $i=0,\dots,t-1$, and hence on telescoping
$$\sum_{\ell \in L} |\ell| \leq \sum_{\ell \in L} |\ell \cap P_t| + d(k-1) t |L| + (k-1) m^{k-1} n_1.$$

\begin{claim}
Let $x$ be a point in $P_t$. Then it has at least $ \frac{m^t}{(k d^k)^{t-1}}$ distinct descendants in $P$.
\end{claim}
\begin{proof}
By definition of $P_t$ there is some $(k-1)$-descendant $x' \in P_{t-1}$ of $x$ which has at least $ m$ children in $P_{t-1}$. Let $S_{t-1} \subseteq P_{t-1}$ denote the set of children of $x'$, so $|S_{t-1}| \geq m$.
By reverse induction for $i= t-1, t-2, \ldots, 0$ we choose sets $S_{i} \subseteq P_i$ of descendants of $x$ so that $|S_{i-1}| \geq \frac{|S_i| m }{k d^k}$. Then $|S_0| \geq  \frac{ m^t}{(k d^k)^{t-1}} $, as wanted.

Let $S_i$ be given. By definition of $P_{i}$ and pigeonhole principle, there is some $0 \leq s \leq k-1$ and $S'_i \subseteq S_i$ such that $|S'_i| \geq \frac{|S_i|}{k}$ and  every $y \in S'_i$ has a strict $s$-descendant $z_y \in P_{i-1}$ with at least $m$ children in $P_{i-1}$.
Fix a path $I_y$ of length $s$ connecting $y$ to $z_y$, and for $0 \leq r \leq s$ let $z^r_y$ denote the $r$th element on the path $I_y$ (so $z_y^0 = y $, $z_y^s = z_y$ and $z_y^{r+1}$ is a child of $z_y^r$). Let $I^r := \{ z^r_y :  y \in S'_i\}$, so $I^0 = S'_i$. Then $|I^{r+1}| \geq \frac{|I^r|}{d}$ (otherwise there is some element $z \in I^{r+1}$ which has at least $d+1$ different parents in $I^r$, which would then form an antichain of size $d+1$ contradicting local $d$-linearity of $P$). 
Hence 
$$|I^s| \geq \frac{|I^0|}{d^{s}}  \geq \frac{|S'_i|}{d^{k-1}} \geq \frac{|S_i|}{k d^{k-1}}.$$ Now by hypothesis 
every element in $I^s$ has at least $m$ children in $P_{i-1}$, denote the set of all the children of the elements in $I^s$ by $S_{i-1} \subseteq P_{i-1}$. Then, again by $d$-linearity, $|S_{i-1}| \geq \frac{|I^s| m}{d} \geq \frac{|S_i| m }{k d^k}$.
\end{proof}

Thus if we choose $m, t$ so that
$$ \left( \frac{m}{kd^k} \right)^t > n_1$$
then we will get a contradiction, unless $P_t$ is empty.  We conclude, for such $m$ and $t$, that
$$\sum_{\ell \in L} |\ell| \leq  d(k-1) t |L| + (k-1) m^{k-1} n_1.$$
If we take $m := \left( \frac{c \log(100+n_1)}{\log \log (100+n_1)} \right)^{\frac{1}{k-1}}$ and $t$ to be the integer part of $\frac{c \log(100+n_1)}{\log \log (100+n_1)}$, and assume that $c$ is sufficiently large relatively to $k$ and $d$, then the claim follows.

\subsection{Reduction for dyadic rectangles}

\begin{defn}\label{def: dyadic box}
	\begin{enumerate}
		\item Define a \emph{dyadic interval} to be a half-open interval $I$ of the form $I = [s2^t, (s+1)2^t)$ for integers $s,t$; we use $|I| = 2^t$ to denote the length of such an interval.
		\item Define a \emph{dyadic box} in $\mathbb{R}^d$ (\emph{dyadic rectangle} when $d=2$) to be a product $I_1 \times \ldots \times I_d$ of dyadic intervals.
	\end{enumerate}
\end{defn}

  Note that if two dyadic intervals intersect, then one must be contained in the other. 

\begin{thm}\label{thm: upper bound for dyadic rects}  Fix $k \in \mathbb{N}$. Assume we have a collection $P$ of $n_1$ points in $\RR^2$ and a collection $R$ of $n_2$ dyadic rectangles in $\RR^2$, 
 with the property that the incidence graph contains no $K_{k,k}$, and $n = n_1 + n_2$.  Then the number of incidences $(p,I \times J)$ with $p \in P$ and $p \in I \times J \in R$ is at most \[ O_k \left( n \frac{\log(100+n_1)}{\log\log(100+n_1)} \right). \]
\end{thm}

\begin{proof}  Suppose that we have some nested dyadic rectangles $D_1 \supseteq D_2 \supseteq \ldots \supseteq D_k$ in $R$. As the incidence graph is $K_{k,k}$-free by hypothesis,  $D_k$ may contain at most $(k-1)$ points from $P$. Removing all such rectangles repeatedly we loose only $(k-1) n_2$ incidences, and thus may assume that any nested sequence in $R$ is of length at most $k-1$. In particular, any rectangle can be repeated at most $k-1$ times in $R$. Then, possibly increasing the number of incidences by a multiple $(k-1)$, we may assume that there are no repetitions in $R$.

We now define a relation $\leq$ on $R$ by declaring $I \times J \leq I' \times J'$ if $I \subseteq I'$ and $J \supseteq J'$. This is a locally $(k-1)$-linear partial order (by the previous paragraph: antisymmetry holds as there are no repetitions in $R$, and using that all rectangles are dyadic, any antichain of size $k$ inside an interval would give a nested sequence of rectangles of length~$k$).

For each point $p$ in $P$, let $\ell_p$ be a subset of $R$ consisting of all those rectangles in $R$ that contain $p$; then $\ell_p$
 is a $(k-1)$-linear set (again, any antichain gives a nested sequence of rectangles of the same length). Finally, $p \in R \iff R \in \ell_p$, hence the collection $\{ \ell_p : p \in P \}$ is a $K_{k,k}$-free arrangement and the claim now follows from Theorem \ref{main} with $d := k-1$.
\end{proof}

\begin{rem}\label{rem: good incidences}
  For a non-dyadic rectangle $R$, let $0.99 R$ denote the rectangle with the same center as R, but whose lengths and heights have been shrunk by a factor of $0.99$.  Define a ``good incidence'' to be a pair $(p,R)$ where $p$ is a point lying in $0.99 R$, not just in $R$.  Then the dyadic bound in Theorem \ref{thm: upper bound for dyadic rects} implies that for a family of arbitrary (not necessarily dyadic) rectangles with no $K_{k,k}$'s, one still gets the $O \left(\frac{n \log n}{ \log \log n} \right)$-type bound for the number of good incidences.  
  
  The reason is as follows.  First we can randomly translate and dilate (non-isotropically, with the horizontal and vertical coordinates dilated separately) the configuration of points and rectangles by some translation parameter and a pair of dilation parameters $(s,t)$ for each of the coordinates. While there is no invariant probability measure on the space of dilatations, one can for instance pick a large number $N$ (much larger than the number of points and rectangles, etc.), dilate horizontally by a random dilation between $1/N$ and $N$ (using say the $dt/t$ Haar measure) making (with positive probability) the horizontal side length close to a power of two; then a vertical dilation will achieve a similar effect for the vertical side length; and then translate by a random amount in $[-N,N]^2$ (chosen uniformly at random) placing the rectangle very close to a dyadic one with positive probability.  If $R$ is a rectangle that is randomly dilated and translated this way, then with probability $>10^{-10}$, there will be a dyadic rectangle $R'$ stuck between $R$ and $0.99 R$.  If the original rectangles have no $K_{k,k}$, then neither will these new dyadic rectangles.  The expected number of incidences amongst the dyadic rectangles is at least $10^{-10}$ times the number of good incidences amongst the original rectangles.  Hence any incidence bound we get on dyadic rectangles implies the corresponding bound for good incidences for non-dyadic rectangles (losing a factor of $10^{10}$).
\end{rem}

\section{A connection to model-theoretic linearity}\label{sec: loc mod}
In this section we obtain a stronger bound in Theorem \ref{thm:main} (without the logarithmic factor) under a stronger assumption that the whole semilinear relation $X$ is $K_{k, \ldots,k}$-free (Corollary \ref{cor: semilin wo log}). And we show that if this stronger bound doesn't hold for a given semialgebraic relation, then the field operations can be recovered from this relation (see Corollary \ref{cor: rec mult from failed zar} for the precise statement). These results are deduced in Section \ref{sec: appl of loc mod} from a more general model-theoretic theorem proved in Section \ref{sec: loc mod gen}.

\subsection{Main theorem}\label{sec: loc mod gen}

We recall some standard model-theoretic notation and definitions, and refer to \cite{marker2006model} for a general introduction to model theory, and to \cite{berenstein2012weakly} for further details on geometric structures.

Recall that $\acl$ denotes the algebraic closure operator, i.e.~if $\mathcal{M} = (M, \ldots)$  is a first-order structure, $A \subseteq M$ and $a$ is a finite tuple in $M$, then $a \in \acl(A)$ if it belongs to some finite $A$-definable subset of $M^{|a|}$ (this generalizes linear span in vector spaces and algebraic closure in fields). Throughout this section we follow the standard model theoretic notation: depending on the context, writing $BC$ denotes either the union of two subsets $B,C$ of $M$, or the tuple obtained by concatenating the (possibly infinite) tuples $B,C$  of elements of $M$.

\begin{defn}
A complete first-order theory $T$ in a language $\mathcal{L}$ is \emph{geometric} if for any model $\mathcal{M} = (M, \ldots) \models T$ we have the following.
\begin{enumerate}
	\item The algebraic closure in $\mathcal{M}$ satisfies the \emph{Exchange Principle}:
	
	\noindent if $a,b$ are singletons in $\mathcal{M}$, $A \subseteq M$ and $b \in \acl(A,a) \setminus \acl(A)$, then $a \in \acl (A,b)$.
	\item $T$ \emph{eliminates $\exists^{\infty}$ quantifier}:
	
	\noindent for every $\mathcal{L}$-formula $\varphi(x,y)$ with $x$ a single variable and $y$ a tuple of variables there exists some $k \in \mathbb{N}$ such that for every $b \in M^{|y|}$, if $\varphi(x,b)$ has more than $k$ solutions in $M$, then it has infinitely many solutions in $M$.
\end{enumerate}  
\end{defn}
 In models of a geometric theory, the algebraic closure operator $\acl$ gives rise to a matroid, and given $a$ a finite tuple in $M$ and $A \subseteq M$, $\dim(a/A)$ is the minimal cardinality of a subtuple $a'$ of $a$ so that $\acl(a \cup A) = \acl(a' \cup A)$ (in an algebraically closed field, this is just the transcendence degree of $a$ over the field generated by $A$). Finally, given a finite tuple $a$ and sets $C,B \subseteq M$, we write $a \ind_C B$ to denote that $\dim \left(a/BC \right) = \dim \left(a / C \right)$.

\begin{rem}
If $T$ is geometric, then it is easy to check that $\ind$ is an \emph{independence relation}, i.e.~ it satisfies the following properties for all tuples $a,a', b,b',d$ and $C,D \subseteq M$:
\begin{itemize}
\item $a \ind_C b \iff \acl(a,C) \ind_C \acl(b,C)$;
\item (extension) if $a \ind_C b$ and $d$ is arbitrary, then there exists some $a'$ so that $a' \ind_C bd$ and $a' \equiv_{Cb} a$ (which means that $a'$ belongs to exactly the same $Cb$-definable subsets of $M^{|a|}$ as $a$).
\item (monotonicity) $a a' \ind_C b b' \implies a \ind_C b$;
\item (symmetry) $a \ind_C b \implies b \ind_C a$;
\item (transitivity) $a \ind_D bb' \iff a \ind_{Db} b' $ and $a \ind_{D} b$;
	\item (non-degeneracy) if $a \ind_C b$ and $d \in \acl(a,C) \cap \acl(b,C)$, then $d \in \acl(C)$.
\end{itemize}
\end{rem}

The following property expresses that the matroid defined by the algebraic closure is \emph{linear}, in the sense that the closure operator behaves more like span in vector spaces, as opposed to algebraic closure in fields.

\begin{defn}\cite[Definition 2.1]{berenstein2012weakly}\label{def: weak loc mod}
A geometric theory $T$ is \emph{weakly locally modular} if 	for any saturated $\mathcal{M} \models T$ and $A,B$ small subsets of $\mathcal{M}$ there exists some small set $C \ind_{\emptyset} AB$ such that $A \ind_{\acl(AC) \cap \acl(BC)} B$.
\end{defn}

Recall that a linearly ordered structure $\mathcal{M}=(M,<, \ldots)$ is \emph{$o$-minimal} if every definable subset of $M$ is a finite union of intervals (see e.g.~ \cite{van1998tame}).
\begin{sample}\label{ex: lin o-min} \cite[Section 3.2]{berenstein2012weakly}
		An $o$-minimal structure is linear (i.e.~any
normal interpretable family of plane curves in $T$ has dimension $\leq 1$)  if and only if it is weakly locally modular.

	In particular, every theory of an ordered vector space over an ordered division ring is weakly locally modular (so Theorem \ref{thm: non-cart lin} applies to semi-linear relations).
\end{sample}

The following is a key model-theoretic lemma.

\begin{lem}\label{lem: loc mod}
Assume that $T$ is geometric and weakly locally modular, and $\mathcal{M} = (M, \ldots) \models T $ is $\aleph_1$-saturated. Assume that $E \subseteq M^{d_1} \times \ldots \times M^{d_r}$ is an $r$-ary relation defined by a formula with parameters in a finite tuple $b$, and $E$ contains no $r$-grid $A=\prod_{i \in [r]}A_i$ with each $A_i \subseteq M^{d_i}$ infinite. Then for any $(a_1, \ldots, a_r) \in E$ there exists some $i \in [r]$ so that $a_i \in \acl \left( \left\{a_j : j \in [r] \setminus \{i\} \right\}, b\right)$.
\end{lem}
\begin{proof}[Proof of Lemma \ref{lem: loc mod}]

Assume not, then there exist some $(a_1, \ldots, a_r)$ in $\mathcal{M}$ such that $(a_1,\ldots, a_r) \in E$, but $a_i \notin \acl \left( a_{\neq i}, b \right)$ for every $i \in [r]$, where $a_{\neq i} := \left\{a_j : j \in [r] \setminus \{i\} \right\}$.

By weak local modularity, for each $i \in [r]$ there exists some small set $C_i \subseteq \mathcal{M}$ so that  
$$C_i \ind_{\emptyset} \left\{a_1, \ldots, a_r \right\} \cup \{b\} \textrm{ and } a_i \ind_{\acl(a_i, C_i) \cap \acl(a_{\neq i}, b, C_i)} a_{\neq i} b.$$

By extension of $\ind$, we may assume that $C_i \ind_{\emptyset} a_1, \ldots, a_r, b, C_{<i}$ for all $i \in [r]$. Hence by transitivity $C \ind_{\emptyset} a_1, \ldots, a_r, b$, where $C := \bigcup_{i \in [r]} C_i$.

Let $D := \bigcap_{i \in [r]} \acl \left( a_{\neq i}, b,  C\right)$.

\begin{claim*}[A]
For every $i \in [r]$, $a_i \ind_D a_{\neq i}$.	
\end{claim*}
\begin{proof}
Fix $i \in [r]$. As $C \ind_{\emptyset} a_1, \ldots, a_r, b$ and $a_i \ind_{\acl(a_i, C_i) \cap \acl(a_{\neq i}, b, C_i)} a_{\neq i} b$, by symmetry and transitivity we have 
$$a_i \ind_{\acl(a_i, C_i) \cap \acl(a_{\neq i}, b, C_i)} a_{\neq i} b C.$$

Note that $\acl(a_i,C_i) \subseteq \acl(a_{\neq j},C)$ for every $i \neq j \in [r]$, hence $\acl(a_i,C_i) \cap \acl(a_{\neq i},b,C_i) \subseteq D$, and clearly $D \subseteq \acl(a_{\neq i},b,C)$. Hence 
$a_i \ind_{D} a_{\neq i} b C$, and in particular $a_i \ind_{D} a_{\neq i}$.
\end{proof}

\begin{claim*}[B]
For every $i \in [r]$, $a_i \notin \acl(D)$.	
\end{claim*}
\begin{proof}
Fix $i \in [r]$. Then $\acl(D) \subseteq \acl(a_{\neq i}, b, C)$ by definition.  But as $C \ind_{a_{\neq i} b} a_i$ by transitivity,  if $a_i \in \acl(a_{\neq i}, b, C)$ then we would get $a_i \in \acl(a_{\neq i},b)$, contradicting the assumption.
\end{proof}

By induction we will choose sequences of tuples $\bar{\alpha}_i = (a_i^{t})_{t \in \mathbb{N}}, i \in [r]$ in $\mathcal{M}$ such that for every $i \in [r]$ we have:
\begin{enumerate}
	\item $a^t_i \equiv_{D \bar{\alpha}_{<i} a_{>i}} a_i$ for all $t \in \mathbb{N}$;
	\item $a^t_i \neq  a_i^{s}$ (as tuples) for all $s \neq t \in \mathbb{N}$;
	\item $\bar{\alpha}_i \ind_D \bar{\alpha}_{<i} a_{>i}$.
\end{enumerate}

Fix $i \in [r]$, and assume that we already chose some sequences $\bar{a}_j$ for $1 \leq j < i$ satisfying (1)--(3). 

\begin{claim*}[C]
	We have $a_i \ind_D \bar{\alpha}_{<i} a_{>i}$.
\end{claim*}
\begin{proof}
If $i=1$, this claim becomes $a_i \ind_D a_{\neq i}$, hence holds by Claim (A). So assume $i \geq  2$.  We will show by induction that for each $l = 1, \ldots, i-1$ we have
$$\bar{\alpha}_{i-1} \ldots \bar{\alpha}_{i-l} \ind_D \bar{\alpha}_{< i-l} a_{>i-1}.$$
For $l =1$ this is equivalent to $\bar{\alpha}_{i-1} \ind_D \bar{\alpha}_{<i-1} a_{> i-1}$, which holds by (3) for $i-1$. So we assume this holds for $l < i-1$, that is we have
$\bar{\alpha}_{i-1} \ldots \bar{\alpha}_{i-l} \ind_D \bar{\alpha}_{< i - l} a_{> i - 1}$,
 and show it for $l+1$. By assumption and transitivity we have 
 $$\bar{\alpha}_{i-1} \ldots \bar{\alpha}_{i-l} \ind_{D \bar{\alpha}_{i-(l+1)}} \bar{\alpha}_{< i - (l+1)}a_{>i-1}.$$
 Also $\bar{\alpha}_{i-(l+1)} \ind_{D} \bar{\alpha}_{< i - (l+1)}a_{>i-1}$ by (3) for $i-(l+1) < i$.  Then by transitivity again $\bar{\alpha}_{i-1} \ldots \bar{\alpha}_{i-l}  \bar{\alpha}_{i-(l+1)}\ind_{D} \bar{\alpha}_{< i - (l+1)}a_{>i-1}$, which concludes the inductive step.
 
 In particular, for $l = i-1$ we get $\bar{\alpha}_{<i} \ind_{D} a_{>i-1}$, that is $\bar{\alpha}_{<i} \ind_{D} a_i a_{>i}$. By transitivity and Claim (A) this implies $\bar{\alpha}_{<i} a_{>i} \ind_{D} a_i$, and we conclude by symmetry.
\end{proof}

Using Claim (C) and extension of $\ind$, we can choose a sequence 
$\bar{\alpha}_i = (a^t_i)_{t \in \mathbb{N}}$ so that $a^t_i \equiv_{D \bar{\alpha}_{<i} a_{>i}} a_i$ and $a^t_i \ind_{D} \bar{\alpha}_{<i} a_{>i} a_i^{<t}$ for every $t \in \mathbb{N}$. By Claim (B) we have $a_i \notin \acl(D)$, hence $a_i^t \notin \acl(D)$, hence $a^t_i \notin \acl\left(\bar{\alpha}_{<i}, a_{>i}, a_i^{<t} \right)$, so in particular all the tuples $(a^t_i)_{t \in \mathbb{N}}$ are pairwise-distinct and $\bar{\alpha}_i$ satisfies (1) and (2). By symmetry and transitivity of $\ind$ we get $\bar{\alpha}_i \ind_{D} \bar{\alpha}_{<i} a_{>i}$. This concludes the inductive step in the construction of the sequences.

Finally, as (1) holds for all $i \in [r]$ and $b$ is contained in $D$, it follows that $(a^{t_1}_{1}, \ldots, a^{t_r}_r) \equiv_{b} (a_1, \ldots, a_r)$, and hence $(a^{t_1}_{1}, \ldots, a^{t_r}_r) \in E$ for every $(t_1, \ldots, t_r) \in \mathbb{N}^r$. By (1) each of the sets $\left\{a^{t}_i : t \in \mathbb{N} \right\}, i \in [r]$ is infinite --- contradicting the assumption on $E$. This concludes the proof of the lemma.
\end{proof}

\begin{thm} \label{thm: non-cart lin}
Assume that $T$ is a geometric, weakly locally modular theory, and $\mathcal{M} \models T$. Assume that $r \in \mathbb{N}_{\geq 2}$ and $\varphi(\bar{x}_1, \ldots, \bar{x}_r,\bar{y})$ is an $\mathcal{L}$-formula without parameters, with $|\bar{x}_i| = d_i, |\bar{y}| = e$. Then there exists some $\alpha = \alpha(\varphi) \in \mathbb{R}_{>0}$ satisfying the following.

Given $b \in M^{e}$, consider the $r$-ary relation 
$$E_b := \left\{(a_1, \ldots, a_r) \in M^{d_1} \times \ldots \times M^{d_r} : \mathcal{M} \models \varphi(a_1, \ldots, a_r, b) \right\}.$$
Then for every $b \in M^e$, exactly one of the following two cases must occur: 
\begin{enumerate}
	\item $E_b$ is not  $K_{k, \ldots, k}$-free for any $k \in \mathbb{N}$;
	\item for any finite $r$-grid $B \subseteq \prod_{i \in [r]} M^{d_i}$ we have
\[ |E_b \cap B| \leq \alpha \delta^r_{r-1}(B). \] 
\end{enumerate}
\end{thm}
\begin{proof}

Assume that $\mathcal{N} = (N, \ldots )$ is an elementary extension of $\mathcal{M}$ and $b \in M^{e}$. Then, for a fixed $k \in \mathbb{N}$,
$$E_b = \{(a_1, \ldots, a_r) \in M^{d_1} \times \ldots \times M^{d_r} : \mathcal{M} \models \varphi(a_1, \ldots, a_r, b) \}$$
is $K_{k, \ldots, k}$-free if and only if 
$$E'_b = \{(a_1, \ldots, a_r) \in N^{d_1} \times \ldots \times N^{d_r} : \mathcal{N} \models \varphi(a_1, \ldots, a_r, b) \}$$
is $K_{k, \ldots, k}$-free, as this can be expressed by a first-order formula $\psi(y)$ and $\mathcal{M} \models \psi(b) \iff \mathcal{N} \models \psi(b)$. Similarly, for a fixed $\alpha \in \mathbb{R}$, $|E_b \cap B| \leq \alpha \delta^r_{r-1}(B)$ for every finite $r$-grid $B \subseteq \prod_{i \in [r]} M^{d_i}$ if and only if $|E'_b \cap B| \leq \alpha \delta^r_{r-1}(B)$ for every finite $r$-grid $B \subseteq \prod_{i \in [r]} N^{d_i}$ (as for every fixed sizes of $B_1, \ldots, B_r$ this condition can be expressed by a first-order formula). Hence, passing to an elementary extension, we may assume that $\mathcal{M}$ is $\aleph_1$-saturated.

As $T$ eliminates $\exists^{\infty}$, there exists some $m = m(\varphi)  \in \mathbb{N}$ such that for any $i \in [r]$, $b \in M^e$ and tuple $\bar{a} := \left( a_j \in M^{d_j} : j \in [r] \setminus \{i\} \right)$, the fiber 
$$E^i_{\bar{a}; b} := \left\{  a^* \in M^{d_i} : \mathcal{M} \models \varphi(a_1, \ldots, a_{i-1}, a^*, a_{i+1}, \ldots, a_{r}; b)\right\}$$
is finite if and only if it has size $\leq m$.

Given an arbitrary $b \in M^{e}$ such that $E_b$ is $K_{k, \ldots, k}$-free, Lemma \ref{lem: loc mod} and compactness imply that for every tuple $(a_1, \ldots, a_r) \in E_b$, there exists some $i \in [r]$ such that the fiber $E^i_{\bar{a}; b}$ is finite, hence $|E^i_{\bar{a}; b}| \leq m$.
This easily implies that for any finite $r$-grid $B \subseteq \prod_{i \in [r]} M^{d_i}$
we have $|E_b\cap B| \leq m \delta^r_{r-1}(B)$.
\end{proof}

\begin{rem}
In the binary case, a similar observation was made by Evans in the context of certain \emph{stable} theories \cite[Proposition 3.1]{evans2005trivial}.	
\end{rem}

Restricting to \emph{distal} structures, we can relax the assumption ``$E_b$ is $K_{k, \ldots, k}$-free for some $k$'' to ``$E_b$ does not contain a direct product of infinite sets'' in Theorem \ref{thm: non-cart lin}  (we refer to e.g.~the introduction in \cite{chernikov2018regularity} or \cite{chernikov2020cutting} for a general discussion of model-theoretic distality and its connections to combinatorics).
	
\begin{cor}\label{cor: distal wo log}

Assume that $T$ is a distal, geometric, weakly locally modular theory, $\mathcal{M} \models T$, $r \in \mathbb{N}_{\geq 2}$ and $\varphi(\bar{x}_1, \ldots, \bar{x}_r,\bar{y})$ is an $\mathcal{L}$-formula without parameters, with $|\bar{x}_i| = d_i, |\bar{y}| = e$. Then there exists some $\alpha = \alpha(\varphi) \in \mathbb{R}_{>0}$ satisfying the following.

Assume that $b \in M^{e}$ and the $r$-ary relation $E_b$ does not contain an $r$-grid $A = \prod_{i \in [r]}A_i$ with each $A_i \subseteq M^{d_i}$ infinite. Then $|E_b \cap B| \leq \alpha \delta^r_{r-1}(B)$ for any finite $r$-grid $B$.
\end{cor}
\begin{proof}
	By \cite[Theorem 5.12]{chernikov2020cutting}, if $\mathcal{M}$ is a distal structure with elimination of $\exists^{\infty}$, then there exists some $k = k(\varphi) \in \mathbb{N}$ such that for every $b \in M^{e}$, $E_b$ is not $K_{k, \ldots, k}$-free if and only if $\prod_{i \in [r]} A_i \subseteq E_b$ for some infinite $A_i \subseteq M^{d_i}$. The conclusion now follows by Theorem \ref{thm: non-cart lin}.
\end{proof}

\begin{rem}
	Weaker bounds for non-cartesian relations definable in arbitrary distal theories are established in \cite{chernikov2018model, chernikov2021model}.\end{rem}

Now we show that in the $o$-minimal case, this result actually characterizes weak local modularity. By the trichotomy theorem in $o$-minimal structures \cite{peterzil1998trichotomy} we have the following equivalence.

\begin{fact}\label{fac: o-min trich}
Let $\mathcal{M}$ be an $o$-minimal ($\aleph_1$-)saturated structure.	The following are equivalent:
\begin{itemize}
	\item $\mathcal{M}$ is not linear (see Example \ref{ex: lin o-min});
	\item $\mathcal{M}$ is not weakly locally modular;
	\item there exists a real closed field definable in $\mathcal{M}$.
\end{itemize}
\end{fact}

\begin{cor}\label{cor: lin iff loc mod} Let $\mathcal{M}$ be an $o$-minimal structure. The following are equivalent:
\begin{enumerate}
\item $\mathcal{M}$ is weakly locally modular;
	\item Corollary \ref{cor: distal wo log} holds in $\mathcal{M}$;
	\item for every $d_1,d_2 \in \mathbb{N}$ and every definable (with parameters) $X \subseteq M^{d_1} \times M^{d_2}$, if $X$ is $K_{k,k}$-free for some $k \in \mathbb{N}$,
then there exist some $\beta < \frac{4}{3}$ and $\alpha$ such
that: for any $n$ and $B_i \subseteq M^{d_i}$ with $|B_i| = n$ we have
\[ |X \cap B_1 \times B_2| \leq \alpha n^{\beta}; \] 
	\item there is no infinite field definable in $\mathcal{M}$.

\end{enumerate}
	
\end{cor}
\begin{proof}
	(1) $\Rightarrow$ (2) by Corollary \ref{cor: distal wo log}, and (2) $\Rightarrow$ (3) is obvious. 
	
	(3) $\Rightarrow$ (4) Assume that $\mathcal{R}$ is an infinite field definable in $\mathcal{M}$, $\textrm{char}(\mathcal{R}) = 0$ by $o$-minimality.
	 Then the point-line incidence relation on $\mathcal{R}^2$ corresponds to a $K_{2,2}$-free definable relation $E \subseteq \mathcal{M}^{d} \times \mathcal{M}^d$ for some $d$. By the standard lower bound for Szemer\'edi-Trotter, the number of incidences satisfies $\Omega(n^{4/3})$, hence $E$ cannot satisfy (3).
	
	(4) $\Rightarrow$ (1) If $\mathcal{M}$ is not weakly locally modular, by Fact \ref{fac: o-min trich} a real closed field $\mathcal{R}$ is definable in $\mathcal{M}$.
\end{proof}

\subsection{Applications to semialgebraic relations}\label{sec: appl of loc mod}

\begin{cor}\label{cor: semilin wo log}
Assume that $X \subseteq \mathbb{R}^d = \prod_{i \in [r]} \mathbb{R}^{d_i}$ is semilinear and $X$  does not contain a direct product of $r$ infinite sets (e.g.~if $X$ is $K_{k, \ldots, k}$-free for some $k$). Then  there exists some $\alpha = \alpha(X)$ so that for any $r$-hypergraph $H$ of the form $\left(V_1, \ldots, V_r; X \cap \prod_{i \in [r]} V_i \right)$ for some finite $V_i \subseteq \mathbb{R}^{d_i}$, with $\sum_{i =1}^r |V_i|=n$, we have $|E| \leq \alpha n^{r-1}$.
\end{cor}
\begin{proof}
	As every $o$-minimal structure is distal and every semilinear relation is definable in an ordered vector space over $\mathbb{R}$ which is $o$-minimal and locally modular (Example \ref{ex: lin o-min}), the result follows by Corollary \ref{cor: distal wo log}.
\end{proof}

We recall the following special case of the trichotomy theorem in o-minimal structures restricted to semialgebraic relations.

\begin{fact}\cite[Theorem 1.3]{marker1992additive} \label{fac: mini trich}
	Let $X \subseteq \mathbb{R}^n$ be a semialgebraic but not semilinear set. Then $\times \restriction_{[0,1]^2}$ (i.e.~the graph of multiplication restricted to the unit box) is definable in the first-order structure $(\mathbb{R},<,+,X)$.
\end{fact}

Using it, we have the following more explicit variant of Corollary \ref{cor: lin iff loc mod} in the semialgebraic case.

\begin{cor}\label{cor: rec mult from failed zar}
	Let $X \subseteq \mathbb{R}^d$ be a semialgebraic set, and consider the first-order structure $\mathcal{M} = (\mathbb{R},<,+,X)$. Then the following are equivalent.
	\begin{enumerate}
		\item For any $r \in \mathbb{N}$ and any $r$-ary relation $Y \subseteq \prod_{i \in [r]}\mathbb{R}^{d_i}$ not containing an $r$-grid $A = \prod_{i \in [r]}A_i$ with each $A_i \subseteq \mathbb{R}^{d_i}$ infinite, there exists some $\alpha \in \mathbb{R}$ so that $|Y \cap B| \leq \alpha \delta^r_{r-1}(B)$ for every finite $r$-grid $B$.
\item For every $d_1,d_2 \in \mathbb{N}$ and $Y \subseteq \mathbb{R}^{d_1} \times \mathbb{R}^{d_2}$ definable (with parameters) in $\mathcal{M}$, if  $Y$ is $K_{k,k}$-free for some $k \in \mathbb{N}$,
then there exist some $\beta < \frac{4}{3}$ and $\alpha$ such
that: for any $n$ and $B_i \subseteq \mathbb{R}^{d_i}$ with $|B_i| = n$ we have
\[ |X \cap B_1 \times B_2| \leq \alpha n^{\beta}. \] 
		\item $\times \restriction_{[0,1]^2}$ is not definable in $\mathcal{M}$.
	\end{enumerate}
\end{cor}
\begin{proof}
(1) $\Rightarrow$ (2) is obvious. 

(2) $\Rightarrow$ (3) Using $\times \restriction_{[0,1]^2}$, the $K_{2,2}$-free point-line incidence relation in $\mathbb{R}^2$ is definable restricted to $[0,1]^2$, and the standard configurations witnessing  the lower bound in Szemer\'edi-Trotter can be scaled down to the unit box.

(3) $\Rightarrow$ (1)
Assume that (1) does not hold in $(\mathbb{R},<,+,X)$. Then necessarily some $Y$ definable in $(\mathbb{R},<,+,X)$ is not semilinear by Corollary \ref{cor: semilin wo log}. By Fact \ref{fac: mini trich}, if $Y$ is not semilinear then $\times \restriction_{[0,1]^2}$ is definable in the structure $(\mathbb{R},<,+,Y)$, hence in $(\mathbb{R},<,+,X)$.
\end{proof}

\bibliography{refs}
    
\end{document}